\documentclass{article}

\usepackage{amsmath}
\usepackage{amssymb}
\usepackage{amsthm}
\usepackage{enumerate}
\usepackage{amscd}
\usepackage{url}
\usepackage{tikz}
\usepackage{tikz-cd}
\usetikzlibrary{arrows}
\usetikzlibrary{calc}
\usepackage[T1]{fontenc}
\usepackage[utf8]{inputenc}

\theoremstyle{plain}
\newtheorem{theorem}{Theorem}
\newtheorem{lemma}[theorem]{Lemma}
\newtheorem{proposition}[theorem]{Proposition}
\newtheorem{corollary}[theorem]{Corollary}

\theoremstyle{definition}
\newtheorem{definition}[theorem]{Definition}
\newtheorem{remark}[theorem]{Remark}

\newtheorem{example}[theorem]{Example}

\newcommand{\beq}{\begin{equation}}
\newcommand{\eeq}{\end{equation}}
\newcommand{\beqv}{\begin{equation*}}
\newcommand{\eeqv}{\end{equation*}}

\newcommand{\mysubsection}[1]{%
 {\rmfamily\bfseries\upshape #1.}%
 \addcontentsline{toc}{section}{\vspace{-1.5\baselineskip}\\ $\qquad\;$\emph{#1}}
 \nopagebreak
 \par
}

\newcommand{\Z}{\mathbb{Z}}

\newcommand{\R}{\mathbb{R}}

\newcommand{\F}{\mathbb{F}}
\newcommand{\PP}{{\mathbb{P}}}
\newcommand{\cI}{{\mathcal{I}}}

\newcommand{\cM}{{\mathcal{M}}}
\newcommand{\ie}{\emph{i.e. }}
\newcommand{\eg}{\emph{e.g. }}
\newcommand{\longto}{\longrightarrow}
\newcommand{\inj}{\hookrightarrow}

\newcommand{\tens}{\otimes}
\newcommand{\moins}{\setminus}
\newcommand{\meet}{\wedge}
\newcommand{\join}{\vee}

\newcommand{\equivaut}{\Longleftrightarrow}

\newcommand{\abs}[1]{\lvert #1\rvert}
\newcommand{\norme}[1]{\lVert #1\rVert}
\newcommand{\linspan}[1]{\langle #1\rangle}
\newcommand{\mumax}{\mu_{\max}}
\newcommand{\mumin}{\mu_{\min}}

\DeclareMathOperator{\rk}{rk}

\DeclareMathOperator{\sat}{sat}
\DeclareMathOperator{\cl}{cl}
\DeclareMathOperator{\covol}{covol}
\DeclareMathOperator{\Supp}{Supp}
\DeclareMathOperator{\Cosupp}{Cosupp}
\DeclareMathOperator{\dmin}{d_{min}}
\DeclareMathOperator{\ratef}{R}
\DeclareMathOperator{\Proj}{Proj}
\DeclareMathOperator{\Div}{Div}
\DeclareMathOperator{\Fil}{Fil}
\DeclareMathOperator{\gr}{gr}

\let\phi\varphi
\let\epsilon\varepsilon
\let\subset\subseteq
\let\supset\supseteq
\let\subsetneq\varsubsetneq
\let\supsetneq\varsupsetneq

\begin{document}

\title{Harder-Narasimhan theory for linear codes\\ {\normalsize \vspace{-.2\baselineskip}(with an appendix on Riemann-Roch theory)}}

\author{Hugues Randriambololona\\
LTCI / ENST ``Telecom ParisTech''}


\maketitle

\begin{abstract}
In this text we develop some aspects of Harder-Narasimhan theory, slopes, semistability and canonical filtration, in the setting of combinatorial lattices.
Of noticeable importance is the Harder-Narasimhan structure associated to a Galois connection between two lattices.
It applies, in particular, to matroids.

We then specialize this to linear codes.
This could be done from at least three different approaches: using the sphere-packing analogy, or the geometric view, or the Galois connection construction just introduced.
A remarkable fact is that these all lead to the same notion of semistability and canonical filtration.
Relations to previous propositions toward a classification of codes, and to Wei's generalized Hamming weight hierarchy, are also discussed.

Last, we study the two important questions of the preservation of semistability (or more generally the behaviour of slopes) under duality, and under tensor product.
The former essentially follows from Wei's duality theorem for higher weights---and its matroid version---which we revisit in an appendix,
developing analogues of the Riemann-Roch, Serre duality, Clifford, and gap and gonality sequence theorems.
Likewise the latter is closely related to the bound on higher weights of a tensor product, conjectured by Wei and Yang, and proved by Schaathun in the geometric language, which we reformulate directly in terms of codes. From this material we then derive semistability of tensor product.

\end{abstract}



\tableofcontents

\section*{Introduction}
\addcontentsline{toc}{section}{Introduction}

A powerful guiding principle in arithmetic geometry is the analogy between number fields and function fields.
This analogy already manifests itself when one considers ``linear algebra'' over these fields:
it then translates into common properties shared by euclidean or hermitian
lattices over rings of integers of number fields, and vector bundles over curves.
The main motivation of the present work is to emphasize how some of these similarities extend to linear codes.

This will certainly be no surprise for anyone familiar with the subject.
Indeed, many connections are known between codes and lattices on one side, and between codes and curves on the other.
However, there might be a deeper explanation for this phenomenon.

Shortly 
said, a linear code is just a subspace of a finite dimensional $L^1$ space over a trivially valued field.
Following the philosophy of Arakelov theory, putting a metric on a linear object can be seen as a replacement for an integral structure.
From this point of view, one could then argue that coding theory is nothing but linear algebra over a certain combinatorial base object of dimension one, hence of arithmetic nature.

\vspace{.5\baselineskip}

The corpus of results that we would like to borrow from number theory
and algebraic geometry into coding theory consists of two parts: first of all,
Harder-Narasimhan theory, and to a lesser extent, Riemann-Roch theory.

\vspace{.5\baselineskip}

Harder and Narasimhan \cite{HN1975} first introduced slopes and the canonical filtration as a tool in the computation of the number of points (and more generally the \'etale cohomology) of certain moduli spaces over finite fields.
The analogue for euclidean and hermitian lattices was then developed by Stuhler~\cite{Stuhler1976} and Grayson~\cite{Grayson1984}.
The theory subsequently saw multiple extensions, and propositions toward a more unified understanding were made,
especially in a categorical framework \cite{Andre2009}\cite{Chen2010}.

Our aim is more modest. We will completely ignore functoriality issues, and focus on the more elementary combinatorial aspects.
For this we first develop the theory in the context of combinatorial lattices\footnote{In this work we will encounter two distinct notions of ``lattices'' which, quite unfortunately, English terminology does not distinguish.
The first one, \emph{r\'eseau} in French, means a discrete subgroup in a continuous space; for us it will always be endowed with a euclidean or hermitian structure.
The second, \emph{treillis} in French, is a poset with meet and join; again we will often add an adjective, such as ``combinatorial'',
in order to help make the distinction.

(Actually, French terminology cannot claim superiority there, since in turn, \emph{r\'eseau} is used for two distinct notions, one that translates in English as lattice, the other as network; and \emph{treillis} also is used for two distinct notions, one that translates as lattice, the other as trellis!)}:
what makes Harder-Narasimhan theory work is the existence of certain features, such as the so-called second isomorphism theorem, and also rank and degree functions satisfying Grayson's parallelogram constraint.
It turns out these are well captured by a classical notion in combinatorial lattice theory, that of (semi)modularity.
This allows us to present Harder-Narasimhan theory under minimalist hypotheses.\footnote{Added in the revised version:
this approach is so natural that, almost simultaneously with this work, it was also introduced in~\cite{Cornut}---but then the focus shifts to other directions.}
As an illustration, we then explain how a Galois connection between two lattices canonically produces a Harder-Narasimhan structure on them.

Specializing to codes, the most natural approach is perhaps to take inspiration from the analogy with euclidean lattices, where the degree function is $-\log$ of the covolume of a sublattice.
We claim that the right analogue of the covolume is the support size of a subcode.
This allows us to define slopes, the canonical filtration of a code, and a notion of semistability,
using the degree function
\beqv
\deg(C')=n-w(C')
\eeqv
for a subcode $C'$ of a $[n,k]$-code $C$.
However other approaches are possible, for instance, via algebraic geometry using the equivalence with configurations of points in a projective space,
or via our combinatorial lattice construction using a certain natural Galois connection.
A very satisfactory result is that all these approaches lead to the same theory.

Our definition of the degree function from the support size implies a close link with Wei's generalized weight hierarchy $d_i(C)$ from \cite{Wei1991}:
the canonical polygon of $C$ is the upper convex envelope of the set of points
\beqv
\{(i,n-d_i(C)):\;0\leq i\leq k\}.
\eeqv
If this polygon has $N$ sides, the code admits a canonically defined $N$-step filtration
\beqv
0=C_0\subsetneq C_1\subsetneq\cdots\subsetneq C_N=C.
\eeqv
However, for codes used in practical applications, this will essentially always be trivial ($N=1$): these codes are stable, and our construction becomes void.
Thus it is to be expected that true coding specialists will find the present work pointless.
Probably, meaningful uses of the theory would appear only when considering questions involving the class of \emph{all} codes, good and bad.

In most Harder-Narasimhan categories, semistability is preserved under duality.
We show it is true for codes, and more generally, we explain how the slopes and canonical filtration of $C$ and $C^\perp$ are related.

Less clear is the preservation of semistability under tensor product. 
It is known to be true for vector bundles over curves in characteristic zero,
but false in positive characteristic, where counter-examples have been constructed \cite{Gieseker1973}.
For euclidean and hermitian lattices the question has been popularized by Bost and is still open,
despite recent progress \cite{BC2013}\cite{GR2013} that allow to settle low dimensional cases, or lattices admitting a large group of automorphisms.
We show it holds for codes, where it is closely related to Schaathun's lower bound on higher weights of a tensor product.
Schaathun's proof~\cite{Schaathun2000} is written in the geometric language, in terms of projective systems of points.
We reformulate (and somehow simplify) the proof, directly in terms of codes, and then derive semistability of tensor product, 
with the hope that it could provide inspiration for new advances in the euclidean lattice situation.

\vspace{.5\baselineskip}

The Riemann-Roch and Serre duality theorems certainly form the core of the theory of algebraic curves, on which almost all other results are based.
In particular the Riemann-Roch theorem is closely related to the functional equation of the zeta function of a curve \cite{Schmidt1931}\cite[\S5.1]{Stichtenoth}.
Analogues for euclidean and hermitian lattices over integers of number fields have been developed at various levels of sophistication,
based on the functional equation of the theta function~\cite{Tate1950},
or using the language of Arakelov theory~\cite{Szpiro1985}\cite{GS1991}\cite{Soule1997},
or a combination of both~\cite{Roessler1993}\cite{vdGS2000}\cite{Groenewegen2001}\cite{Borisov2003}\cite{Bost2015} (references in chronological order).

Linear codes also admit zeta functions with a functional equation~\cite{Duursma1999},
and we can see the MacWilliams relation for the weight enumerator as an analogue of the functional equation of the theta function.
This suggests that some sort of Riemann-Roch theory should exist also for codes.
And indeed, in an appendix that can be read independently,
we propose a definition of $H^0$ and $H^1$ spaces for codes that satisfy analogues of Serre duality and Riemann-Roch (see Theorems~\ref{Serre} and~\ref{RR}), and seem to fit well in this framework.
Maybe these notions are not entirely new: previous occurrences can be hinted, for instance, in~\cite{Duursma2003}.
However our version is more precise and makes the analogy very explicit.

This Riemann-Roch theory can serve as a nice way to present the behaviour of slopes and canonical filtration under duality.
It can also be used to reprove Wei's duality theorem for higher weights~\cite{Wei1991}.

\vspace{.5\baselineskip}

Closely related to linear codes is the notion of matroid.
Whenever possible, we tried to explain how our theory extends in this framework.
Matroids also admit slopes and a canonical filtration, hence a notion of semistability.
They satisfy a form of Riemann-Roch theory, from which duality results for their slopes and canonical filtration can be derived, extending those for codes.
We do not try to extend the notion of tensor product to matroids, although this certainly is an interesting question.

Satisfactorily, matroid theory recently gained interest among algebraic and arithmetic geometers (we refer to \cite{Katz} for an introduction to this topic),
especially thanks to its interplay with non-archimedean and tropical geometry.
Relevant in this context is the Riemann-Roch theorem for graphs of \cite{BN2007}, and its many further generalizations.
To any graph there is an associated matroid, however, it is unclear whether our Riemann-Roch theorem for matroids and this Riemann-Roch theorem for graphs could be related.
Last, a very elaborate cohomology theory for matroids was recently introduced in \cite{AHK2015}.
Again, to which extent our work could possibly relate to it remains unclear.

\vspace{.5\baselineskip}

We will borrow some notations commonly used in combinatorial theory. We let $[n]=\{1,\dots,n\}$ be the standard set with $n$ elements.
Given a set $S$, we identify $S^n$ with $S^{[n]}$. We also let $2^S$ stand for the set of all subsets of $S$, and $\binom{S}{n}$ for the set of its subsets of cardinality $n$.

Most of the content of this text was known to the author since at least half a dozen years.
Impetus for writing it down and making it publicly available was provided by ANR-14-CE25-0015 project \textsc{Gardio} and ANR-15-CE39-0013 project \textsc{Manta}.


\section{Harder-Narasimhan theory for combinatorial lattices}

\mysubsection{Basic definitions}

We recall that a combinatorial lattice $(L,\subset,\meet,\join)$ is a poset in which any two elements admit a meet and a join.
Here the symbol $\subset$ stands for an abstract partial order relation, although in many examples it will be set inclusion, hence our notation.
We will use the book \cite{Birkhoff} as our main reference on this topic.

A lattice $L$ is said to be \emph{of finite length} if there is a finite upper bound on the length of chains in $L$. Unless otherwise specified, this will always be the case in this work.
Then $L$ admits a minimal element, say $0_L$, and also a maximal element, say $1_L$.
We define the rank $\rk(x)$ of an element $x\in L$ as the maximal length of a chain from $0_L$ to $x$.
Note that our terminology here departs slightly from \cite{Birkhoff} and the standard literature, where our notion of rank is more commonly called height.

Given $x,y\in L$ with $y\subset x$, we will let $x/y$ stand for the sublattice of $L$ made of all $z$ with $y\subset z\subset x$.
Sometimes, by abuse of notation, it will also be tempting to let $x$ stand for $x/0_L$; we might use this licence occasionally, although we will refrain from doing so too often (observe it makes the notations $y\subset x$ and $y\in x$ interchangeable). Likewise, we might occasionally write $L/y$ for $1_L/y$.


In this work we will only consider \emph{modular} lattices \cite[\S I.7]{Birkhoff}.
There are various equivalent characterizations of this important notion.
We will use the one from \cite[\S II.8, Th.~16]{Birkhoff}.
First we recall that a function $f:L\longto\R$ is said \emph{lower semimodular} if for any $x,y\in L$ we have
\beq
\label{semimodularity}
f(x)+f(y)\leq f(x\join y)+f(x\meet y).
\eeq
Dually $f$ is upper semimodular if $-f$ is lower semimodular.
And $f$ is \emph{modular} if it is both lower and upper semimodular, that is, 
if for any $x,y\in L$ we have
\beq
\label{modularity}
f(x)+f(y)=f(x\join y)+f(x\meet y).
\eeq
Then, a lattice $L$ of finite length is \emph{modular} if and only if it satisfies the following two conditions:
\begin{itemize}
\item the rank function $\rk:L\longto\Z$ is modular, and moreover
\item (Jordan-Dedekind) all maximal chains between two $y\subset x$ in $L$ have the same length, necessarily equal to $\rk(x)-\rk(y)$.
\end{itemize}
It is known that for any $x,y$ in a modular lattice, we have a natural identification $(x\join y)/y\simeq x/(x\meet y)$ \cite[\S I.7, Th.~13]{Birkhoff}.

\begin{definition}
\label{def_HN_lattice}
A Harder-Narasimhan lattice is a modular lattice $L$ of finite length
equipped with a lower semimodular function $\deg:L\longto\R$.
\end{definition}

\begin{lemma}
\label{pullback}
Suppose $(L,\deg)$ is a Harder-Narasimhan lattice, and let $L'$ be another modular lattice of finite length together
with a morphism $L'\longto L$ (meaning that it respects meet and join).
Then pulling back the degree function to $L'$ makes it into a Harder-Narasimhan lattice.

In particular, for any elements $y\subset x$ in $L$, the sublattice $x/y$ is a Harder-Narasimhan lattice.
\end{lemma}
\begin{proof}
Clear.
\end{proof}



To any $x\in L$ we associate a point $p_{x}$ in the real plane, with coordinates $p_{x}=(\rk(x),\deg(x))$.
Graphically, the modularity and semimodularity relations 
\beq
\label{parallelogram}
\begin{aligned}
\rk(x)+\rk(y)&=\rk(x\join y)+\rk(x\meet y)\\
\deg(x)+\deg(y)&\leq\deg(x\join y)+\deg(x\meet y)
\end{aligned}
\eeq
translate into Grayson's parallelogram constraint \cite[Fig.~1.14]{Grayson1984}:
any three points among $p_{x},p_{y},p_{x\meet y},p_{x\join y}$ determine a parallelogram,
and then the fourth point lies vertically below (if it is $p_{x}$ or $p_{y}$) or above (if it is $p_{x\meet y}$ or $p_{x\join y}$) the fourth vertex of this parallelogram.\footnote{Actually, in Grayson's original paper the constraint goes in the opposite direction.
This is because his sign convention for the degree function makes it upper semimodular instead of lower semimodular.}
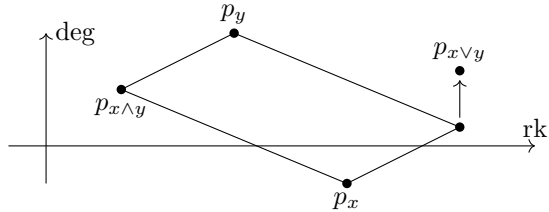
\begin{center}
\begin{tikzpicture}[scale=.25]
\coordinate[label=below:$p_{x\meet y}$] (W1) at (3cm,1cm);
\coordinate[label=above:$p_{y}$] (N1) at (9cm,4cm);
\coordinate (E1) at (21cm,-1cm);
\coordinate[label=below:$p_{x}$] (S1) at (15cm,-4cm);
\coordinate[label=above:$p_{x\join y}$] (Z1) at (21cm,2cm);
\foreach \point in {W1,N1,E1,S1,Z1} \fill (\point) circle (2.5mm);
\draw (W1) -- (N1) -- (E1) -- (S1) -- (W1); 
\draw[->] ($(E1)+(0mm,5mm)$) -- ($(Z1)+(0mm,-5mm)$);
\draw[->] (-3cm,-2cm) -- (25cm,-2cm);
\coordinate[label=above:$\rk$] (axex) at (25cm,-2cm);
\draw[->] (-1cm,-4cm) -- (-1cm,4cm);
\coordinate[label=right:$\deg$] (axey) at (-1cm,4cm);
\end{tikzpicture}
\end{center}

If $L$ is nontrivial we define its slope as $\mu(L)=(\deg(1_L)-\deg(0_L))/\rk(L)$ where by abuse of notation we set $\rk(L)=\rk(1_L)$. 
Accordingly, given $y\subsetneq x$ in $L$, the slope from $y$ to $x$ is
\beq
\mu(x/y)=\frac{\deg(x)-\deg(y)}{\rk(x)-\rk(y)}
\eeq
that is, the slope of the line segment joining $p_{y}$ to $p_{x}$.
We also define the slope of $x\neq0_L$ to be $\mu(x/0_L)$ and the co-slope of $x\neq1_L$ to be $\mu(1_L/x)$.

\begin{lemma}
Let $(L,\deg)$ be a Harder-Narasimhan lattice. Then the degree function $\deg$ is bounded from above.
\end{lemma}
\begin{proof}
We proceed by induction on the length $n=\rk(L)$ of $L$. The result is obvious if $n=0$ or $1$,
so we assume $L$ has length $n\geq2$ and make the following
\emph{Induction hypothesis}: on every Harder-Narasimhan lattice of length at most $n-1$, the degree function
is bounded from above.

By contradiction suppose $\deg$ is not bounded on $L$, and choose $r$ maximal such that the set
$$E_r=\{\deg(x):x\in L,\,\rk(x)=r\}$$
is unbounded
(observe $r\leq n-1$ since $E_n=\{\deg(1_L)\}$ is finite).

Now fix an $a\in L$ of rank $\rk(a)=1$, and let $x\in L$ vary with $\rk(x)=r$ and $\deg(x)$ arbitrarily large.
The sublattice $L/a$ has length $n-1$, so by our \emph{Induction hypothesis} we will have $x\not\in L/a$
as soon as $\deg(x)$ is large enough. This means $a\not\subseteq x$, and forces $a\meet x=0_L$
and $$\rk(a\join x)=r+1.$$
But then, by semimodularity,
$$\deg(a\join x)\geq\deg(x)+\deg(a)-\deg(0_L)$$
can be arbitrarily large when $\deg(x)$ is.
This means precisely that $E_{r+1}$ is unbounded, contradicting the maximality of $r$.
\end{proof}

This boundedness allows us to define a concave, piecewise linear function
\beq
P_L:[0,\rk(L)]\longto\R
\eeq
as the infimum of all linear functions whose graph lies above the set $\{p_{x}:\,x\in L\}$.
The graph of $P_L$ is then called the \emph{canonical polygon} of $L$.

Let $N$ be the number of sides of this polygon, and label its vertices $(i_\alpha,P_L(i_\alpha))$, for $0\leq\alpha\leq N$, according to increasing $\rk$-coordinate
\beq
0=i_0<i_1<\dots<i_N=\rk(L).
\eeq
Note that by construction these $i_\alpha$ are integers.

\begin{definition}
The successive slopes of $L$ are the slopes of its canonical polygon.
\end{definition}
Actually, depending on the authors, there are two conflicting labeling schemes for these successive slopes.
One first variant would be to define the $i$-th slope of $L$ as $P_L(i)-P_L(i-1)$, for $1\leq i\leq\rk(L)$. This gives a non-increasing sequence,
with repetition allowed.
The other variant, which is the one we will adopt, is to extract only the set of distinct values from this sequence.
So, for $1\leq\alpha\leq N$, we define the $\alpha$-th slope of $L$ as
\beq
\mu_\alpha(L)=P_L(i_\alpha)-P_L(i_\alpha-1)=\frac{P_L(i_\alpha)-P_L(i_{\alpha-1})}{i_\alpha-i_{\alpha-1}}.
\eeq
This gives a strictly decreasing sequence. We also set $\mumax=\mu_1$ and $\mumin=\mu_N$.
By construction we then have
\beq
\mumax(L)\geq\mu(L)=\frac{1}{\rk(L)}\sum_{1\leq\alpha\leq N}(i_\alpha-i_{\alpha-1})\mu_\alpha(L)\geq\mumin(L).
\eeq

Many examples of Harder-Narasimhan lattices satisfy a Northcott-type property, which reads that for any real $B$, there are only finitely many $x\in L$ with $\deg(x)\geq B$.
Thus, under this hypothesis, for any $i$, the supremum $\sup_{\rk(x)=i}\deg(x)$ is attained.
In particular, for any $\alpha$, there is an $x_\alpha\in L$ with $\rk(x_\alpha)=i_\alpha$ and $\deg(x_\alpha)=P_L(i_\alpha)$;
that means $p_{x_\alpha}$ is the corresponding vertex of the canonical polygon.
It turns out this last assertion holds \emph{unconditionally}, as will be seen in Theorem~\ref{th_filtr} below. 
%

\begin{lemma}
\label{lemme_th_filtr}
Let $x,y\in L$ satisfy $\rk(y)=i_\beta\leq\rk(x)=i_\alpha$ but $y\not\subset x$.
Suppose $\deg(x)\geq P_L(i_\alpha)-\epsilon$ for some $\epsilon>0$.
Then we have $\deg(y)\leq P_L(i_\beta)-(\mu_\beta(L)-\mu_{\alpha+1}(L))+\epsilon$.
\end{lemma}
\begin{proof}
This is a consequence of the parallelogram constraint \eqref{parallelogram}.
Indeed, since $y\not\subset x$, we have $\rk(x\meet y)<i_\beta$, so we can write $\rk(x\meet y)=i_\beta-k$ with $k\geq1$, and then $\rk(x\join y)=i_\alpha+k$.
By concavity of $P_L$ and by definition of the slopes we deduce $\deg(x\meet y)\leq P_L(i_\beta)-k\mu_\beta(L)$ and $\deg(x\join y)\leq P_L(i_\alpha)+k\mu_{\alpha+1}(L)$.
From semimodularity of the degree we then conclude $\deg(y)\leq\deg(x\meet y)+\deg(x\join y)-\deg(x)\leq P_L(i_\beta)-k(\mu_\beta(L)-\mu_{\alpha+1}(L))+\epsilon$.
\end{proof}

\begin{theorem}[compare {\cite[Th.~1.18]{Grayson1984}}]
\label{th_filtr}
Let $L$ be a Harder-Narasimhan lattice.
\begin{enumerate}[(i)]
\item For any $\alpha$, there is an $x_\alpha\in L$ with $\rk(x_\alpha)=i_\alpha$ and $\deg(x_\alpha)=P_L(i_\alpha)$.
Moreover this $x_\alpha$ is unique, and in fact any $y\neq x_\alpha$ with $\rk(y)=i_\alpha$ is subject to the gap condition
\beqv
\deg(y)\leq P_L(i_\alpha)-(\mu_\alpha(L)-\mu_{\alpha+1}(L)).
\eeqv
\item These $x_\alpha$ form a chain, \ie for $\beta\leq\alpha$ we have $x_\beta\subset x_\alpha$.
\end{enumerate}
\end{theorem}
\begin{proof}
The first assertion follows from Lemma~\ref{lemme_th_filtr} with $\beta=\alpha$ and $\epsilon\to0$.
Likewise for the second assertion, if we set $x=x_\alpha$ and $y=x_\beta$ and take $\epsilon$ small enough, then from $x_\beta\not\subset x_\alpha$ we would get a contradiction in Lemma~\ref{lemme_th_filtr}.
\end{proof}

\begin{definition}
This chain
\beqv
0_L=x_0\subsetneq x_1\subsetneq\dots\subsetneq x_N=1_L,
\eeqv
with $p_{x_\alpha}=(i_\alpha,P_L(i_\alpha))$,
is called the \emph{canonical filtration} of $1_L$ (or of $L$).
\end{definition}

\begin{definition}
We say $L$ is \emph{semistable} if its canonical filtration is trivial, which means $L$ has only one slope ($N=1$), or equivalently, if any $x\neq0_L,1_L$ has slope $\mu(x/0_L)\leq\mu(L)$, or equivalently, co-slope $\mu(1_L/x)\geq\mu(L)$.

Moreover we say $L$ is \emph{stable} if any $x\neq0_L,1_L$ has slope $\mu(x/0_L)<\mu(L)$, or equivalently, co-slope $\mu(1_L/x)>\mu(L)$.
\end{definition}

\begin{proposition}
\label{colle}
Let $L$ be a Harder-Narasimhan lattice.
\begin{enumerate}[(i)]
\item Let $x\in L$, $x\neq0_L,1_L$, with $\mumin(x/0_L)\geq\mumax(1_L/x)$. Then the canonical polygon of $L$ can be obtained by pasting together those of $x/0_L$ and $1_L/x$.
\item Let $x\in L$, $x\neq0_L,1_L$. Then $x$ is part of the canonical filtration of $L$ if and only if $\mumin(x/0_L)>\mumax(1_L/x)$.
\item A chain $0_L=x_0\subsetneq x_1\subsetneq\dots\subsetneq x_N=1_L$ is the canonical filtration of $L$ if and only if all $x_i/x_{i-1}$ are semistable with $\mu(x_{i+1}/x_i)>\mu(x_i/x_{i-1})$.
\end{enumerate}
\end{proposition}
\begin{proof}
Consequence of Theorem~\ref{th_filtr}, as in \cite[Cor.~1.29-1.31]{Grayson1984}.
\end{proof}

All these constructions behave well under certain ``affine'' transformations:
\begin{lemma}
\label{modif_affine}
Let $(L,\deg)$ be a Harder-Narasimhan lattice, with canonical polygon $P_L$ and slopes $\mu_1>\dots>\mu_N$.
\begin{enumerate}[(i)]
\item Let $L^{\textrm{opp}}$ be the opposite lattice of $L$. Then $(L^{\textrm{opp}},\deg)$ is a Harder-Narasimhan lattice.
Its canonical polygon is the image of $P_L$ under the plane transformation $(\xi,\eta)\mapsto(\rk(L)-\xi,\eta)$.
Its canonical filtration is the opposite of that of $L$, with slopes $-\mu_N>\dots>-\mu_1$.
\item Let $a,b\in\R$ and $c\in\R_{>0}$ be constants. Then $(L,a+b\rk+c\deg)$ is a Harder-Narasimhan lattice.
Its canonical polygon is the image of $P_L$ under the plane transformation $(\xi,\eta)\mapsto(\xi,a+b\xi+c\eta)$.
It has the same canonical filtration as $(L,\deg)$, with slopes $b+c\mu_1>\dots>b+c\mu_N$.
\end{enumerate}
As a consequence, as soon as one of these Harder-Narasimhan lattice is semistable (resp. stable), then all of them are.
\end{lemma}
\begin{proof}
Clear.
\end{proof}
We will say $(L,\deg)$ is normalized if $\deg(0_L)=0$. We will say it is co-normalized if $\deg(1_L)=0$.
Thanks to the second part of the Lemma, we see that it is always possible to modify the degree function in order to force one, or even both of these conditions, 
without changing the essential properties of the lattice.

\begin{example}
\label{exmatroid}
We recall that a matroid \cite{Whitney1935}\cite{Oxley} is a pair $\cM=(E,\cI)$ where $E$ is a set, and $\cI\subset 2^E$ a collection of subsets of $E$ (the \emph{independent} sets), such that:
\begin{itemize}
\item $\emptyset\in\cI$;
\item if $I\in\cI$ and $I'\subset I$, then $I'\in\cI$;
\item if $I_1,I_2\in\cI$ and $\#I_1<\#I_2$, then there is some $e\in I_2\moins I_1$ such that $I_1\cup\{e\}\in\cI$.
\end{itemize}
For any subset $J\subset E$ (not necessarily independent) we set
\beq
\label{defr}
r(J)=\max\{\#I:\;I\in\cI,\,I\subset J\}.
\eeq
It is easily seen \cite[Lemma~1.3.1]{Oxley} that this function $r$ is upper semimodular.
As a consequence, if $E$ is finite, the lattice $L_E$ of all subsets of $E$ is modular of finite length, and becomes a Harder-Narasimhan lattice thanks to the degree function
\beq
\label{degmatroid}
\deg(J)=k-r(J)
\eeq
where $k$ could be any arbitrary constant; we will set $k=r(E)$, which makes $\deg$ nonnegative and co-normalized.

We observe a quite unfortunate conflict between the well-established terminology of the domain and ours.
Indeed, $r(J)$ is traditionally called the rank of $J$, while for us the rank of $J$ in $L_E$ is $\#J$.
Worse, some authors in matroid theory call $\#J$ the degree of $J$!
An explanation for this inversion of terms will come from Remark~\ref{Galoismatroids} below.
\end{example}

\begin{example}
\label{exHN}
Let $E$ be a vector bundle, \ie a locally free sheaf of finite rank, on a projective curve over a field.
Given a subsheaf $F\subset E$, we let $F^{\sat}$ be the smallest subsheaf of $E$ containing $F$ such that $E/F^{\sat}$ is locally free.
We say $F$ is saturated if $F=F^{\sat}$. The saturated subsheaves of $E$ form a modular lattice $L_E$ of finite length,
whose meet and join are given by $F\meet F'=F\cap F'$ and $F\join F'=(F+F')^{\sat}$. Then, the usual degree function for vector bundles
is lower semimodular on $L_E$, so it makes it into a Harder-Narasimhan lattice.
Its canonical filtration was first introduced by Harder and Narasimhan \cite{HN1975} as a tool in the computation of the number of points (and more generally the \'etale cohomology) of certain moduli spaces over finite fields.
\end{example}

\begin{example}
\label{exSG}
Let $E$ be a euclidean lattice, \ie a free $\Z$-module of finite rank, equipped with a positive definite scalar product.
Given a submodule $F\subset E$, we let $F^{\sat}$ be the smallest submodule of $E$ containing $F$ such that $E/F^{\sat}$ is torsion-free.
We say $F$ is saturated if $F=F^{\sat}$. The saturated submodules of $E$ form a modular lattice $L_E$ of finite length,
whose meet and join are given by $F\meet F'=F\cap F'$ and $F\join F'=(F+F')^{\sat}$.
Then, the arithmetic degree function $\deg(F)=-\log\covol(F)$ is lower semimodular on $L_E$, so it makes it into a Harder-Narasimhan lattice.
The analogy with Example~\ref{exHN} was first noticed by Stuhler \cite{Stuhler1976} and used to revisit certain results in reduction theory.
It was then generalized by Grayson \cite{Grayson1984} in order to study the cohomology of arithmetic groups, by constructing manifolds with boundary on which the group naturally acts, alternative to those of \cite{BS1973}.
\end{example} 

\begin{remark}
\label{genHN}
In both Examples~\ref{exHN} and~\ref{exSG} the rank and degree functions on $L_E$ naturally extend to the lattice $L'_E$ of all (possibly non-saturated) submodules of $E$, and they still satisfy the parallelogram constraint~\eqref{parallelogram}.
As a consequence, although $L'_E$ is not of finite length so our Definition~\ref{def_HN_lattice} does not apply, slopes and the canonical filtration are still well defined for $L'_E$, and they coincide with those of $L_E$.
In fact the map $F\mapsto F^{\sat}$ from $L'_E$ onto $L_E$ is a morphism of lattices, and it is interesting to note that the degree function becomes modular when restricted to its fibers.
Continuing in this spirit it is possible to generalize our theory to lattices that are not of finite length.
Moreover, adding suitable topological conditions, it is even possible to allow the degree function to be $\R$-valued instead of $\Z$-valued.
We will not elaborate on these ideas since they are not needed in the sequel.
\end{remark}

\begin{remark}
In most Harder-Narasimhan categories (in the sense of~\cite{Chen2010}), the lattice of subobjects of an object is a Harder-Narasimhan lattice in our sense.
However, this fails for \cite[Ex.~2]{Chen2010}, the category of vector spaces with two hermitian norms.
Actually, the degree function in this category does not satisfy our semimodularity condition, which~\cite{Chen2010} omitted in its definitions.
This omission leads to certain problems, for instance, the proof of \cite[Prop.~4.5]{Chen2010} is incorrect
(the equality in \cite[p.~195, l.~10]{Chen2010} holds only if the exact sequence lies in $\mathcal{E}_A$, which is not assumed).
A corrected version had been planned~\cite{Chen2016}.

Still, vector spaces with two norms enjoy certain properties very close to ours, and it would be intesting to investigate whether a weaker version of Harder-Narasimhan theory could be developed in order to reflect this.
\end{remark}

\vspace{.5\baselineskip}
\mysubsection{Modular lattices under a Galois connection}

We recall \cite[\S V.8]{Birkhoff} that a \emph{Galois connection} between two lattices $L$ and $M$ is a pair of \emph{order-reversing}\footnote{
So this is sometimes called an ``antitone'' Galois connection.
There is the alternative notion of a ``monotone'' Galois connection, in which the maps are order-preserving, leading to an essentially equivalent theory.}
maps $(.)^\circ:L\to M$ and $(.)^\circ:M\to L$ such that any $x$ in $L$ or $M$ satisfies
\beq
\label{closure}
x\subset x^{\circ\circ}.
\eeq
%
It then follows that $x^\circ=x^{\circ\circ\circ}$, hence the map $x\mapsto\overline{x}=x^{\circ\circ}$ is a \emph{closure operator} on $L$ and on $M$,
\ie it satisfies $x\subset\overline{x}$ and $\overline{x}=\overline{\overline{x}}$.
We say $x$ is \emph{closed} if $x=\overline{x}$,
which happens precisely when $x$ is of the form $x=a^\circ$ for some~$a$ (\eg $a=x^\circ$).

It is also easily shown that for any $l\in L$ and $m\in M$ we have
\beq
\label{lm-ml}
l\subset m^\circ\quad\equivaut\quad m\subset l^\circ,
\eeq
and in fact one could show conversely that this property entirely characterizes Galois connections,
\ie that \eqref{lm-ml} implies the conditions in the definition.

\begin{proposition}
\label{inGal}
Let $L,M$ be given with such a Galois connection, and take $x,y\in L$ or $x,y\in M$. Then we have
\beq
\label{joinGal}
x^\circ\join y^\circ\subset(x\meet y)^\circ
\eeq
and
\beq
\label{meetGal}
x^\circ\meet\,y^\circ=(x\join y)^\circ.
\eeq
\end{proposition}
\begin{proof}
Since $(.)^\circ$ is order-reversing we get $x^\circ\subset(x\meet y)^\circ$ and $y^\circ\subset(x\meet y)^\circ$, hence~\eqref{joinGal}.
Likewise we get $x^\circ\supset(x\join y)^\circ$ and $y^\circ\supset(x\join y)^\circ$, hence
\beq
\label{halfmeetGal}
x^\circ\meet\,y^\circ\supset(x\join y)^\circ
\eeq 
which is half of~\eqref{meetGal}.

On the other hand by \eqref{closure} and \eqref{joinGal} we have
\beq
\label{beuargh}
x\join y\;\,\subset\;x^{\circ\circ}\join y^{\circ\circ}\;\subset\;(x^\circ\meet\,y^\circ)^\circ.
\eeq
Applying \eqref{lm-ml} we deduce $x^\circ\meet\,y^\circ\subset(x\join y)^\circ$ and conclude.
\end{proof}
As a consequence we get that the subset $L^{\cl}$ of closed elements in $L$ is a lattice with meet $x\meet y$ and join $\overline{x\join y}$, so the natural map $\overline{(.)}:L\to L^{\cl}$ becomes a morphism of lattices
(although we will make no use of it, it is amusing to note the analogy with the map $(.)^{\sat}:L'_E\to L_E$ of Remark~\ref{genHN}).
Moreover, $(.)^\circ$ then becomes an anti-isomorphism between the lattices $L^{\cl}$ and $M^{\cl}$.

\begin{theorem}
\label{thGal}
Let $L,M$ be two modular lattices of finite length.
Let $\rk_L$ be the natural rank function of $L$, and $\rk_M$ that of $M$.
Suppose given a Galois connection $(.)^\circ$ between $L$ and $M$.
Then the degree functions
\beqv
\deg_L(l)=\rk_M(l^\circ)
\eeqv
and $\deg_M(m)=\rk_L(m^\circ)$, for $l\in L$ and $m\in M$, make $L$ and $M$ into Harder-Narasimhan lattices.
\end{theorem}
\begin{proof}
The roles of $L$ and $M$ being symmetric, it suffices to show that $\deg_L$ is lower semimodular.
And indeed, for $x,y\in L$ we find
\beq
\label{preuvesousmodularite}
\begin{split}
\deg_L(x)+\deg_L(y)&=\rk_M(x^\circ)+\rk_M(y^\circ)\\
&=\rk_M(x^\circ\join y^\circ)+\rk_M(x^\circ\meet y^\circ)\\
&\leq\rk_M((x\meet y)^\circ)+\rk_M((x\join y)^\circ)\\
&=\deg_L(x\meet y)+\deg_L(x\join y)
\end{split}
\eeq
where we used that $\rk_M$ is modular and the intermediate inequality comes from Proposition~\ref{inGal}.
\end{proof}

Because of the identity $x^{\circ\circ\circ}=x^\circ$, we observe that all $l\in L$ satisfy
\beq
\label{degbar=deg}
\deg_L(\overline{l})=\deg_L(l)
\eeq
and likewise $\deg_M(\overline{m})=\deg_M(m)$ for all $m\in M$.

\vspace{\baselineskip}

It is easily seen (say from \eqref{lm-ml}) that we have $0_L^\circ=1_M$ and $0_M^\circ=1_L$, hence $1_L$ and $1_M$ are closed.
From this we also find $1_M^\circ=\overline{0}_L$ and $1_L^\circ=\overline{0}_M$, which motivates the following:

\begin{definition}
We say that $L$ is separated (with respect to the given Galois connection) if $1_M^\circ=0_L$, or equivalently, if $0_L=\overline{0}_L$ is closed.
And we say the Galois connection is separated if both $L$ and $M$ are.
\end{definition}

Any Galois connection between $L$ and $M$ induces (by restriction) a separated Galois connection between $L/\overline{0}_L$ and $M/\overline{0}_M$.
So by Theorem~\ref{thGal} it defines a degree function on $L/\overline{0}_L$ (resp. on $M/\overline{0}_M$), which is just a translate of the restriction of the degree function of $L$ (resp. of $M$).
Moreover, an element $x$ in $L/\overline{0}_L$ (resp. in $M/\overline{0}_M$) is closed if and only if it is closed in $L$ (resp. in $M$).

\begin{lemma}
\label{lemsepGal}
We have $\mumax(L)\leq0$ and $\deg_M(1_M)\geq0$. Moreover, the following assertions are equivalent:
\begin{itemize}
\item $\mumax(L)=0$
\item $\deg_M(1_M)>0$
\item $L$ is not separated.
\end{itemize}
If any of these assertions holds, the first nonzero element of the canonical filtration of $L$ is $\overline{0}_L$, with slope $0$,
and rank $\rk_L(\overline{0}_L)=\deg_M(1_M)$.
And then the subsequent elements of the canonical filtration of $L$, and its subsequent slopes, are those of $L/\overline{0}_L$ (relative to the induced separated Galois connection).
\end{lemma}
\begin{proof}
%
Let $0_L=x_0\subsetneq x_1\subsetneq\dots\subsetneq x_N=1_L$ be the canonical filtration of $L$.
Then $x_1^\circ\subset x_0^\circ$, so $\deg_L(x_1)=\rk_L(x_1^\circ)\leq\rk_L(x_0^\circ)=\deg_L(x_0)$, hence $\mu_1\leq0$.

Moreover if $\mu_1=0$, which means $\deg_L(x_1)=\deg_L(x_0)$, then necessarily $x_1^\circ=x_0^\circ=1_M$,
hence $1_M^\circ\supset x_1$, and $\deg_M(1_M)=\rk_L(1_M^\circ)\geq\rk_L(x_1)>0$.
In turn if $\deg_M(1_M)=\rk_L(1_M^\circ)>0$, then $\overline{0}_L=1_M^\circ\supsetneq 0_L$, so $L$ is not separated.
Last, if $L$ is not separated, then $\mu_1\geq\mu(\overline{0}_L/0_L)=0$, so $\mu_1=0$.

This shows the equivalence of the three conditions.

Now if any of them holds, say the first one, then $x_1$ is the largest element of $L$ with $x_1^\circ=1_M$, which means $x_1=1_M^\circ=\overline{0}_L$.
And then the relation between the canonical filtration of $L$ and that of $L/\overline{0}_L$ follows from the first point in Proposition~\ref{colle}.
\end{proof}

\begin{proposition}
\label{symGal}
Under the hypotheses of Theorem~\ref{thGal}, the canonical polygons of $L$ and $M$ are image of each other under the reflection across the diagonal in the first quadrant.

In particular, if the Galois connection is separated, these canonical polygons are defined, respectively,
by functions
\beqv
P_{L,M}:[0,\rk(L)]\to[0,\rk(M)]
\eeqv
and
\beqv
P_{M,L}:[0,\rk(M)]\to[0,\rk(L)],
\eeqv
that are inverse of each other.

More precisely, suppose that the Galois connection is separated, and let $L$ have canonical filtration
\beqv
0_L=x_0\subsetneq x_1\subsetneq\dots\subsetneq x_N=1_L
\eeqv
with slopes
\beqv
\mu_1>\dots>\mu_N.
\eeqv
Then the $x_\alpha$ are closed, the $\mu_\alpha$ are negative,
and $M$ has canonical filtration
\beqv
0_M=x_N^\circ\subsetneq x_{N-1}^\circ\subsetneq\dots\subsetneq x_0^\circ=1_M
\eeqv
with slopes
\beqv
\mu_N^{-1}>\dots>\mu_1^{-1}.
\eeqv
In particular, $L$ is semistable if and only if $M$ is.
\end{proposition}
\begin{proof}[Proof (abridged)]
Direct consequence of the symmetry between $(L,\deg_L)$ and $(M,\deg_M)$ in Theorem~\ref{thGal}.
\end{proof}
\begin{proof}[Proof (detailed)]
Thanks to Lemma~\ref{lemsepGal}, the first assertion reduces to the second.
So we suppose the Galois connection is separated (and in particular $0_L=x_0$ is closed). By Lemma~\ref{lemsepGal} again, we deduce $\mu_1<0$, hence all $\mu_\alpha<0$.

For any real $\nu>0$ consider the function $\Phi_{L,\nu}$ on $L$ defined by
\beq
\Phi_{L,\nu}(l)=\nu\rk_L(l)+\deg_L(l),
\eeq
which we might also view as the linear form $(\xi,\eta)\mapsto\nu\xi+\eta$ on the real plane, evaluated at the point $p_l=(\rk_L(l),\deg_L(l))$.

Observe the following \emph{Fact}: if $l$ is not closed, then $\rk_L(\overline{l})>\rk_L(l)$, while $\deg_L(\overline{l})=\deg_L(l)$ by~\eqref{degbar=deg}, hence $\Phi_{L,\nu}(\overline{l})>\Phi_{L,\nu}(l)$.

That $\mu_\alpha$ is the $\alpha$-th  slope of the canonical polygon of $L$ means that there is a certain constant $c_\alpha$ such that all $l\in L$ satisfy
\beq
\label{eq_droite}
\Phi_{L,-\mu_\alpha}(l)\leq c_\alpha
\eeq
and that equality is effectively attained for at least two $l$, the smallest of which being $x_{\alpha-1}$, and the largest $x_\alpha$.

Since $\Phi_{L,-\mu_\alpha}(x_\alpha)=c_\alpha$ but $\Phi_{L,-\mu_\alpha}(\overline{x}_\alpha)>c_\alpha$ is forbidden, then from the \emph{Fact} just above
we deduce that $x_\alpha=\overline{x}_\alpha$ is closed.


Last take $m\in M$, and write~\eqref{eq_droite} in the form
\beq
(-\mu_\alpha)\rk_L(l)+\rk_M(l^\circ)\leq c_\alpha.
\eeq
Setting $l=m^\circ$ we then find
\beq
\rk_L(m^\circ)+(-\mu_\alpha^{-1})\rk_M(m^{\circ\circ})\leq(-\mu_\alpha^{-1})c_\alpha,
\eeq
and since $\rk_M(m)\leq\rk_M(m^{\circ\circ})$ we deduce
\beq
\label{eq_droite'}
\Phi_{M,-\mu_\alpha^{-1}}(m)\leq c'_\alpha
\eeq
with $c'_\alpha=(-\mu_\alpha^{-1})c_\alpha$.
Moreover a necessary and sufficient condition for equality in~\eqref{eq_droite'} is that $l=m^\circ$ reaches equality in~\eqref{eq_droite} and $m=m^{\circ\circ}$ is closed.
Thus we see that equality is attained for at least two $m$, the smallest of which being $x_\alpha^\circ$, and the largest $x_{\alpha-1}^\circ$.
This means precisely that $\mu_\alpha^{-1}$ is a slope of $M$, and $x_\alpha^\circ$ and $x_{\alpha-1}^\circ$ the corresponding successive elements of its canonical filtration.
\end{proof}

\begin{remark}
\label{Galoismatroids}
In the proof of Theorem~\ref{thGal}, we observe that \eqref{preuvesousmodularite} still works if $\rk_M$ is only assumed lower semimodular, instead of modular. We deduce:

\emph{
Let $(.)^\circ$ be a Galois connection between a modular lattice $L$ of finite length and a lower semimodular lattice $M$.
Then the degree function
\beq
\label{degGaloisasym}
\deg_L(l)=\rk_M(l^\circ)
\eeq
makes $L$ into a Harder-Narasimhan lattice.
}

Now let $E$ be a finite matroid, and let $L_{E,\textrm{fl}}$ be its lattice of flats \cite[\S1.7]{Oxley}.
It is known that $L_{E,\textrm{fl}}$ is upper semimodular, so the opposite lattice $L_{E,\textrm{fl}}^{\textrm{opp}}$ is lower semimodular.
Let then $L_E$ be the (modular) lattice of all subsets of $E$. We define a Galois connection between $L_E$ and $L_{E,\textrm{fl}}^{\textrm{opp}}$ as follows:
in one direction, we map any $X\in L_E$ to its closure (or span) in $L_{E,\textrm{fl}}^{\textrm{opp}}$;
in the other direction, we just use the forgetful map, \ie the natural inclusion $L_{E,\textrm{fl}}^{\textrm{opp}}\inj L_E$.
Observe these maps are order reversing, by definition of the opposite lattice.

It is then easily seen that, under this Galois connection, the degree function on $L_E$ given by \eqref{degGaloisasym} coincides with \eqref{degmatroid} in Example~\ref{exmatroid}.
\end{remark}


\section{Linear codes}

From now on we will use the following notations:
\begin{itemize}
\item $n\geq1$ is an integer;
\item $[n]=\{1,\dots,n\}$ is the standard set with $n$ elements
\item $\F$ is a field;
\item $\abs{.}$ is the trivial absolute value on $\F$, so $\abs{0}=0$ and $\abs{x}=1$ for $x\in\F^\times$;
\item $\norme{.}$ is the corresponding $L^1$ norm on the standard $\F$-vector space $\F^n$, also called the Hamming norm.

It follows that for $\mathbf{x}=(x_1,\dots,x_n)\in\F^n$ we have
\beq
\norme{\mathbf{x}}=\abs{x_1}+\cdots+\abs{x_n}=\#\Supp(\mathbf{x})
\eeq
where $\Supp(\mathbf{x})=\{i\in[n]:\,x_i\neq0\}$ is the support of $\mathbf{x}$.
\end{itemize}


Now we recall that a (linear) code of length $n$ and dimension $k$ over $\F$, sometimes called a $[n,k]$-code, is a $k$-dimensional vector subspace
\beq
C\subset\F^n
\eeq
equipped with the induced norm.
Beside its length and dimension, another important parameter of a code is its minimal distance
\beq
\dmin(C)=\lambda_1(C)=\min_{\mathbf{c}\in C\moins\{\mathbf{0}\}}\norme{\mathbf{c}}.
\eeq
A $[n,k]$-code with minimum distance $d$ is also called a $[n,k,d]$-code.

The definition of a code we just gave was oriented in a way to stress some similarity with certain branches of arithmetic geometry (\eg Arakelov, Berkovich, or tropical geometry),
where metrics are introduced in order to reflect a certain integral structure in the objects considered.
However, this analogy is certainly not perfect.
For instance, arithmeticians tend to favor metrics that satisfy the ultrametric inequality.
We observe that in our case, although the trivial absolute value $\abs{.}$ is ultrametric, the Hamming norm $\norme{.}$, as a $L^1$ norm, is \emph{not}.

\vspace{.5\baselineskip}
\mysubsection{First approach: the sphere-packing analogy}

Many deep connections exist between the theory of codes and that of euclidean lattices.
Classical references on this topic are \cite{CS}\cite{Ebeling}.
Perhaps the similarity between these objects becomes the most striking when both are viewed as defining regular packings:
of Hamming spheres in $\F^n$ for codes, and of euclidean spheres in $\R^n$ for lattices.
In both cases, a central problem is then that of the determination of the densest such packings.
However, many other notions and problems, either purely mathematical or more algorithmic, can be formulated in a very similar way in both contexts, and are equally interesting.
Let us just mention Minkowski's successive minima, which generalize the notion of the minimum distance of a lattice,
and whose analogues, the successive minima of our $[n,k]$-code $C$, can be defined as
\beq
\lambda_i(C)=\min\{t>0:\,\dim_\F\linspan{C\cap\mathbf{B}_{\norme{.}}(\mathbf{0},t)}\geq i\}
\eeq
for $1\leq i\leq k$.
On the other hand, as discussed in \cite{Borek2005}\cite{BC2013}\cite{BK2010}, slopes of euclidean lattices are known to be closely related to successive minima, while enjoying better functorial properties.
Our aim here will be to find the analogue of this theory for codes.
That is, we want to define a suitable degree function that makes the lattice $L_C$ of linear subspaces of $C$ into a Harder-Narasimhan lattice, in a way as close as possible to Example~\ref{exSG}.

In Example~\ref{exSG}, the degree of $F\subset E$ is an elementary function of its covolume $\covol(F)$.
When $F$ has rank $1$, this covolume is just the norm of a generating vector.
Now for a subcode $C'\subset C$ of dimension $1$, all generating vectors have the same norm, which is also the cardinality of their common support.
This suggests more generally that the analogue of the covolume for an arbitrary subcode $C'\subset C$ should be the cardinality of its support
\beq
\Supp(C')=\bigcup_{\mathbf{c}\in C'}\Supp(c)=\{i\in[n]:\,\exists\mathbf{c}=(c_1,\dots,c_n)\in C',\,c_i\neq0\}.
\eeq
And then the degree should be an elementary function of this support size.

\begin{definition}
\label{def_deg_C'}
Let $C\subset\F^n$ be a $[n,k]$-code.
We define the degree of a linear subcode $C'\subset C$ by the formula
\beq
\label{formule_def_deg_C'}
\deg(C')=n-w(C')
\eeq
where $w(C')=\#\Supp(C')$ is its support size.
\end{definition}
In fact, since degree functions differing only by an additive constant define the same slopes (Lemma~\ref{modif_affine}), we could equally well have used $-w(C')$ in Definition~\ref{def_deg_C'} instead of $n-w(C')$.
Still, for reasons that will appear later, we will keep this choice, even if a peculiar consequence is that the zero subcode then has degree $\deg(0)=n$.

\begin{proposition}
This degree function $\deg:L_C\longto\{0,\dots,n\}$ is lower semimodular,
so it makes $L_C$ into a Harder-Narasimhan lattice.
\end{proposition}
\begin{proof}
Observe that for linear subcodes $C',C''\subset C$ we have
\beq
\begin{aligned}
\Supp(C'+C'')&=\Supp(C')\cup\Supp(C'')\\
\Supp(C'\cap C'')&\subset \Supp(C')\cap\Supp(C'').
\end{aligned}
\eeq
It follows that $C'\mapsto w(C')$ is upper semimodular, and we conclude.
\end{proof}

We define the \emph{effective rate} of a nonzero subcode $C'\subset C$ as
\beq
\ratef(C')=\frac{\dim_\F(C')}{w(C')}.
\eeq
Observe that if $C$ itself has full support, \ie if $\Supp(C)=[n]$, then $\ratef(C)=k/n$ is its rate in the usual sense. In fact coordinates outside $\Supp(C)$ play no role, so we can always reduce to this case.

\begin{proposition}
\label{slope=-rate-1}
The slope of a nonzero $C'\subset C$ is
\beq
\mu(C'/0)=-\ratef(C')^{-1}.
\eeq
Consequently, $C$ is semistable (resp. stable) iff for all $0\subsetneq C'\subsetneq C$ we have
\beq
\label{ratestable}
\ratef(C')\leq\ratef(C)\quad\text{(resp. $\ratef(C')<\ratef(C)$).}
\eeq
\end{proposition}
\begin{proof}
Direct from the definitions.
\end{proof}
Sometimes it is useful to reformulate \eqref{ratestable} as follows: $C$ is semistable iff all $C'$ satisfy
\beq
\label{poidsstable}
w(C')\geq\frac{\dim(C')}{R(C)},
\eeq
and similarly with $>$ for stability. In particular, the  minimum distance of a stable code satisfies the (admittedly unimpressive) lower bound
\beq
\dmin(C)>\frac{1}{R(C)}.
\eeq

\begin{remark}
In Example~\ref{exSG}, take $F\subset E$ and set $l=\rk(F)$.
Then we have $\covol(F)=\norme{v_1\wedge\dots\wedge v_l}$ where $v_1,\dots,v_l$ are generators of $F$ and $\norme{.}$ is the $l$-th alternate power of the norm of $E$.
Back to codes, we have a natural identification $\bigwedge^l\F^n=\F^{\binom{n}{l}}$, and we can define the $l$-th alternate power of the Hamming norm of $\F^n$ as the Hamming norm of $\F^{\binom{n}{l}}$.
So given $C'\subset C\subset\F^n$ of dimension $\dim(C')=l$, we get a line $\bigwedge^lC'\subset\F^{\binom{n}{l}}$, and the norm of a generator of this line would be another convincing analogue of the notion of covolume.
In fact, this norm is easily seen to be the number of \emph{information sets} of $C'$, so it is at most $\binom{w(C')}{l}$. However, inequality could be strict.
This raises two natural questions: can one use this construction to define another semimodular degree function, hence another Harder-Narasimhan structure on $L_C$?
if so, how does it compare with the one deduced from Definition~\ref{def_deg_C'}?
\end{remark}

\vspace{.5\baselineskip}
\mysubsection{Second approach: the geometric view}

Roughly speaking, the geometric view on coding theory, developed by the Russian school, aims at seeing any code as an evaluation code.
A standard reference on the topic is \cite{TV}.

Let $C\subset\F^n$ be a $[n,k]$-code.
We can view $C$ as an abstract $k$-dimensional vector space equipped with $n$ linear forms $\pi_1,\dots,\pi_n$, where $\pi_i:C\to\F$ is the $i$-th coordinate projection.
It is easily seen that $\pi_1,\dots,\pi_n$ span the dual vector space $C^\vee$.
Now, to make things simpler, we shall assume that $C$ has dual distance at least $3$, which means that these $\pi_i$ are nonzero and pairwise nonproportional.
As a consequence they define $n$ distinct points $\overline{\pi}_1,\dots,\overline{\pi}_n$ in the projective space $\PP(C)=\Proj S^\cdot C\simeq\PP^{k-1}$ that parameterizes hyperplanes in $C$ (or equivalently, lines in $C^\vee$).
One can then show that the set of points
\beq
\Pi=\{\overline{\pi}_1,\dots,\overline{\pi}_n\}\subset\PP(C)
\eeq
uniquely determines $C$ up to linear code equivalence, \ie up to the natural action of $\mathfrak{S}_n\ltimes(\F^\times)^n$, the group of linear isometries of $\F^n$.
This might be thought as a reformulation of the celebrated MacWilliams extension theorem, a proof of which along these lines can be found in \cite[Prop.~3.1]{Assmus1998}.

In case $C$ only has dual distance $2$ (\ie $C$ has full support, so the $\pi_i$ still are nonzero, but some of them can be proportional), then things work the same provided we consider $\Pi$ as a multiset, counting the now possibly nondistinct $\overline{\pi}_i$ according to their multiplicity.

Finite configurations of points in a projective space can be classified thanks to geometric invariant theory (GIT),
which provides in particular its own notion of (semi)stability.
Recall (see \eg \cite[Prop.~3.4]{MFK}, \cite[Ch.~II, Th.~1]{DO}, or \cite[Th.~11.2]{Dolgachev} with all $k_i=1$):

\emph{The configuration of points $\Pi\subset\PP(C)$ is semistable (resp. stable) if and only if for all linear subspaces $\emptyset\neq V\subsetneq\PP(C)$ we have}
\beq
\label{GIT}
\frac{\#(\Pi\cap V)}{\dim(V)+1}\leq\frac{n}{k}\qquad\text{(resp. $<$)}.
\eeq

This could be extended in order to fit into our framework:
\begin{proposition}
\label{HNgeom}
Let $L_{\PP(C)}$ be the lattice of all linear subspaces $V\subset\PP(C)$, including $V=\emptyset$ considered as the unique linear subspace of dimension $-1$. Then:
\begin{enumerate}
\item This lattice $L_{\PP(C)}$ is modular, with rank function $\rk(V)=\dim(V)+1$.
\item The degree function $\deg_\Pi(V)=\#(\Pi\cap V)$ is lower semimodular, hence makes $L_{\PP(C)}$ into a Harder-Narasimhan lattice.
\item Then $(L_{\PP(C)},\deg_\Pi)$ is semistable (resp. stable) in our sense, if and only if $\Pi\subset\PP(C)$ is semistable (resp. stable) in the sense of GIT~\eqref{GIT}.
\end{enumerate}
\end{proposition}
\begin{proof}
Routine verification.
\end{proof}

We observe that, in turn, this construction also provides a Harder-Narasimhan structure on the lattice $L_C$ of linear subcodes of $C$.
Indeed, since we defined $\PP(C)$ as the projective space of hyperplanes in $C$, there is a natural identification
\beq
L_{\PP(C)}=(L_C)^{\textrm{opp}}.
\eeq
We then have only to apply Lemma~\ref{modif_affine}.

\vspace{.5\baselineskip}
\mysubsection{Third approach: the cosupport Galois connection}

Let $C\subset\F^n$ be a $[n,k]$-code.
We construct a Galois connection between the lattice $L_C$ of subcodes of $C$ and the lattice $L_{[n]}$ of subsets of $[n]$, as follows.

\begin{definition}
\label{cosupport_connection}
Given a subcode $C'\subset C$ we set
\beq
\label{cosupport_connection_C}
(C')^\circ=\Cosupp(C')=[n]\moins\Supp(C')
\eeq
the cosupport of $C'$, \ie the set of coordinates on which $C'$ is identically $0$.

Given a subset $J\subset [n]$ we set
\beq
\label{cosupport_connection_n}
J^\circ=C_{[n]\moins J}=C\cap\F^{[n]\moins J}
\eeq
the subcode made of all the codewords of support disjoint from $J$, \ie the largest subcode vanishing identically over $J$.
\end{definition}
It is easily seen that $(.)^\circ:L_C\to L_{[n]}$ and $(.)^\circ:L_{[n]}\to L_C$ indeed define a Galois connection, which we will call the \emph{cosupport Galois connection}.

From this we get a Harder-Narasimhan structure on $L_C$ (and also on $L_{[n]}$) thanks to Theorem~\ref{thGal},
on which Proposition~\ref{symGal} applies.
Separation is treated by the following:

\begin{lemma}
The lattice $L_C$ is always separated under the cosupport Galois connection.
On the other hand, the lattice $L_{[n]}$, hence also the cosupport Galois connection itself, is separated if and only if $C$ has full support.
\end{lemma}
\begin{proof}
The zero subcode satisfies $0^\circ=[n]$ hence $\overline{0}=[n]^\circ=0$.
On the other hand we have $\emptyset^\circ=C$ hence $\overline{\emptyset}=[n]\moins\Supp(C)$.
\end{proof}

As follows from \eqref{cosupport_connection_n}, the Harder-Narasimhan structure just constructed on $L_{[n]}$
is given by the degree function
\beq
\label{deg_n}
\deg(J)=\dim(C_{[n]\moins J}).
\eeq
Actually this can be related to another classical construction.
We can view the inclusion $C\subset\F^n$ as a length $1$ filtration on the vector space $\F^n$,
from which we get (see \eg \cite[Ex.~1]{Chen2010}, or \cite{FW1994}\cite{Totaro1994} for the original works) a degree function $\deg(V)=\dim(C\cap V)$ for $V$ in the lattice $L_{\F^n}$ of all vector subspaces of $\F^n$.
Now mapping $J\subset[n]$ to $\F^J\subset\F^n$ embeds $L_{[n]}$ as a sublattice of $L_{\F^n}$, so this degree function pulls back by Lemma~\ref{pullback}.
Last the involution $J\mapsto[n]\moins J$ identifies $L_{[n]}$ and $(L_{[n]})^{\textrm{opp}}$; applying Lemma~\ref{modif_affine} we retrieve~\eqref{deg_n}.

We mention yet another interpretation, in terms of matroids: if $[n]$ is viewed as the index set of the columns of a given generating matrix $G$ of $C$,
then \eqref{deg_n} coincides with the degree function \eqref{degmatroid} of the corresponding matroid.
Indeed, the row-span of the submatrix $G_J$ of columns of $G$ indexed by $J$ is $\pi_J(C)$, where $\pi_J:\F^n\to\F^J$ is the natural projection,
so $r(J)=\rk(G_J)=\dim\pi_J(C)$, and
\beq
\label{k-r(J)}
k-r(J)=\dim\ker(\pi_J|_C)=\dim(C_{[n]\moins J})
\eeq
as claimed.

\vspace{.5\baselineskip}
\mysubsection{Discussion}

We have just introduced three Harder-Narasimhan structures on the lattice $L_C$ of linear subcodes of $C$:
\begin{itemize}
\item one given by Definition~\ref{def_deg_C'}
\item one deduced from Proposition~\ref{HNgeom}, through the natural identification $L_C=(L_{\PP(C)})^{\textrm{opp}}$ and Lemma~\ref{modif_affine}
\item one deduced from the cosupport Galois connection of Definition~\ref{cosupport_connection}, through Theorem~\ref{thGal}.
\end{itemize}

\begin{proposition}
These three Harder-Narasimhan structures on $L_C$ coincide.
\end{proposition}
\begin{proof}
Keep all the notations introduced before.
It is well known (see \eg the proof of \cite[Th.~2.1]{TV1995}) that if $V\subset\PP(C)$ is a $l$-codimensional linear subspace,
corresponding to a $l$-dimensional subcode $C'\subset C$, then $w(C')=n-\#(\Pi\cap V)$, or
\beq
\#(\Pi\cap V)=n-w(C').
\eeq
This gives equality of the first and second Harder-Narasimhan structures.

On the other hand we have $(C')^\circ=[n]\moins\Supp(C')$ hence
\beq
\rk_{L_{[n]}}((C')^\circ)=\#((C')^\circ)=n-w(C'),
\eeq
which gives equality of the first and the third.
\end{proof}

Relationships between the various mathematical objects encountered in this work
can be summarized by the following diagram:
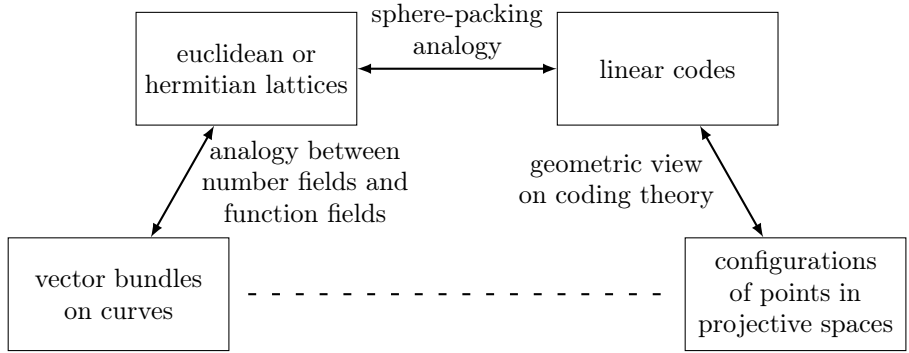
\begin{center}
\begin{tikzpicture}
\node[draw,text width=2.7cm,text centered,minimum width=2.7cm,minimum height=1.5cm] (A) at (-4.5cm,-3cm) {vector bundles on curves};
\node[draw,text width=2.7cm,text centered,minimum width=2.7cm,minimum height=1.5cm] (B) at (-2.8cm,0cm) {euclidean or hermitian lattices};
\node[draw,text width=2.7cm,text centered,minimum width=2.7cm,minimum height=1.5cm] (C) at (2.8cm,0cm) {linear codes};
\node[draw,text width=2.7cm,text centered,minimum width=2.7cm,minimum height=1.5cm] (D) at (4.5cm,-3cm) {configurations of points in projective spaces};
\draw[<->,>=latex,thick] (A) -- (B) node[midway,right,text width=3cm,text centered]{analogy between number fields and function fields};
\draw[<->,>=latex,thick] (B) -- (C) node[midway,above,text width=3cm,text centered]{sphere-packing analogy};
\draw[<->,>=latex,thick] (C) -- (D) node[midway,left,text width=2.8cm,text centered]{geometric view on coding theory};
\draw[shorten <=2.5mm,shorten >=2.5mm,loosely dashed,thick] (A) -- (D);
\end{tikzpicture}
\end{center}
The two extremal objects in this diagram belong to the realm of classical algebraic geometry, and can be classified using GIT.
It is quite remarkable that the associated Harder-Narasimhan structures, and in particular the notions of slopes, canonical filtration, and semistability, are compatible with this chain of analogies.

\begin{remark}[thanks to a suggestion of Bost]
\label{remBost}
In the discussion following \eqref{deg_n} we saw how the degree function on the Galois dual lattice $L_{[n]}$ could be deduced from the formalism of (multi-)filtered vector spaces of Faltings-W\"ustholz and Totaro,
applied to the length $1$ filtration $C\subset\F^n$.
It turns out this theory can also be applied directly to the lattice $L_C$:
consider the collection of length $1$ filtrations $H_i\subset C$, for $1\leq i\leq n$, where $H_i=\ker\pi_i|_C$ is the $i$-th coordinate hyperplane.
This makes $C$ into a multi-filtered vector space, with associated degree function
\beqv
\deg'(V)=\sum_{1\leq i\leq n}\dim(V\cap H_i)
\eeqv
for $V\subset C$.
However we have $\dim(V\cap H_i)=\dim(V)$ if $V\subset H_i$, \ie if $i\not\in\Supp(V)$, and $\dim(V\cap H_i)=\dim(V)-1$ else.
It follows
\beqv
\deg'(V)=n\dim(V)-w(V)=\deg(V)+n\dim(V)-n,
\eeqv
so the two degree functions $\deg$ and $\deg'$ are equivalent up to some affine transformation.
In particular, by Lemma~\ref{modif_affine}, they define the same canonical filtration on $C$.

Moreover, to any multi-filtered vector space (at least in characteristic zero), Faltings-W\"ustholz associate a certain vector bundle on a curve, whose stalks are determined by the filtrations, in such a way that this construction preserves semistability \cite[proof of Th.~4.1]{FW1994}.
Specializing to the case where the filtrations are of length $1$ and defined by hyperplanes, or equivalently, by a configuration of points in the dual projective space,
this fits as the dashed line in our diagram.
\end{remark}


A still largely open problem is that of the classification of codes (say up to linear isometries),
or at least the description of a certain ``structure theory'' from which one could easily identify ``good'' codes \cite{Slepian1960}\cite{Assmus1998}.
Since the geometric view on linear codes makes them equivalent with projective configurations of points, GIT appears as a natural tool toward this goal.
Arguably one could see the present work as a first step in this direction.
Actually, ``zero-th step'' would probably be a more accurate description:
indeed, as could be hinted from~\eqref{poidsstable}, codes used for real world error correction will invariably tend to be stable.
So their canonical filtration, and all our associated constructions, become trivial.
Finer geometric invariants have to be considered.
(Still there are a few other exotic applications of codes beside error correction, and one could wonder whether the canonical filtration has a role to play there.)

\vspace{.5\baselineskip}

Our Harder-Narasimhan theory appears to be decorelated from the notions of criticality and spectrum introduced by Assmus in~\cite{Assmus1998}.
Recall that a code is decomposable (resp. indecomposable) if it can (resp. cannot) be written as the direct sum of two nontrivial subcodes with disjoint supports.
An indecomposable code is critical if ``puncturing'' any coordinate,
\ie projecting onto the remaining $n-1$ positions, produces a decomposable code.
The spectrum of an indecomposable $k$-dimensional code $C$ is then the set of all (isometry classes of) critical $k$-dimensional codes $C'$ that can be obtained by projecting $C$ onto some subset of the coordinates.
A sample of the diversity of possible situations is illustrated by the following examples:
\begin{itemize}
\item The binary $[3,2,2]$-code is stable and critical.
\item The $q$-ary $k$-dimensional simplex code, of length $\frac{q^k-1}{q-1}$, whose associated projective configuration is the set of all points in $\PP^{k-1}(\F_q)$, is stable but not critical as soon as $k\geq3$.
In fact, its spectrum is the set of \emph{all} (classes of) $k$-dimensional critical codes.
\item The binary $[9,7]$-code $C$ with generating matrix
\beqv
\left(\begin{array}{ccccccccc}
1&1&0&0&0&0&0&0&0\\
0&0&1&1&0&0&0&0&0\\
0&0&0&0&1&0&0&0&1\\
0&0&0&0&0&1&0&0&1\\
0&0&0&0&0&0&1&0&1\\
0&0&0&0&0&0&0&1&1\\
0&1&0&1&0&0&0&0&1\\
\end{array}\right)
\eeqv
is critical (it is an instance of \emph{The Construction} of \cite[p.~621]{Assmus1998} with the $3$-partition $X=[9]=\{1,2\}\cup\{3,4\}\cup\{5,6,7,8,9\}$, $X'=\emptyset$)
but not semistable (its subcode $C_{\{5,6,7,8,9\}}$ has rate $4/5>7/9$).
\end{itemize}

\vspace{.5\baselineskip}

Given an (increasing) filtration $\Fil^.$ on an object $X$ in an abelian category, it is customary to define an associated graded object by considering the subsequent quotients $\gr^t X=\Fil^{t}X/\Fil^{t-1}X$.
Unfortunately, the category of linear codes \cite{Assmus1998} is not abelian: in general it lacks a good notion of quotient code.
Still, a nice quotient can be defined for subcodes defined by a support condition: given subsets $S\subset S'\subset[n]$ we define formally $C_{S'}/C_S=\pi_{S'\moins S}(C_{S'})$.
This is compatible with our convention for (combinatorial) lattices, since subcodes of $\pi_{S'\moins S}(C_{S'})$ identify with intermediary subcodes between $C_S$ and $C_{S'}$.

Now let $C$ be a $[n,k]$-code with full support, with canonical filtration
\beqv
0=C_0\subsetneq C_1\subsetneq\cdots\subsetneq C_N=C
\eeqv
and associated slopes $\mu_1>\cdots>\mu_N$.
By Proposition~\ref{symGal} the $C_\alpha$ are closed, \ie they satisfy $C_\alpha=C_{\Supp(C_\alpha)}$, so the above applies and we can define
\beq
\gr^\alpha(C)=\pi_{T_\alpha}(C_\alpha)
\eeq
where $T_\alpha=\Supp(C_\alpha)\moins\Supp(C_{\alpha-1})$.
If $C_\alpha$ has parameters $[n_\alpha,k_\alpha]$, then $\gr^\alpha(C)$ has parameters $[n_\alpha-n_{\alpha-1},k_\alpha-k_{\alpha-1}]$
and is semistable of slope $\mu_\alpha=-\frac{n_\alpha-n_{\alpha-1}}{k_\alpha-k_{\alpha-1}}$.
This gives a way to extract ``possibly good'', \ie semistable codes, from ``assuredly bad'' ones.

The same works for matroids: if $\cM=(E,\cI)$ has canonical filtration
\beqv
\emptyset=J_0\subsetneq J_1\subsetneq\cdots\subsetneq J_N=E
\eeqv
with associated slopes $\mu_1>\cdots>\mu_N$,
we set
\beq
\gr^\alpha(\cM)=(T_\alpha,\cI\cap 2^{T_\alpha})
\eeq
where $T_\alpha=J_\alpha\moins J_{\alpha-1}$. This is a semistable matroid, of slope $\mu_\alpha$.

\vspace{.5\baselineskip}

Last, there is an obvious link between our constructions and Wei's theory of generalized Hamming weights \cite{Wei1991}.
Recall that the $i$-th weight of a (full support) $[n,k]$-code $C$ is
\beq
\label{defdi}
d_i(C)=\min\{w(C'):\;C'\subset C,\,\dim(C')=i\}
\eeq
so $0=d_0(C)<d_1(C)=\dmin(C)<\dots<d_k(C)=n$.
The $C'\subset C$ for which the min is attained in \eqref{defdi} will be called minimum weight subcodes (of dimension $i$).
One should keep in mind that, for a given $i$, there need not be a unique such subcode.
More generally, two arbitrary minimum weight subcodes (of different dimension) need not be included one in the other.
A code is said to satisfy the \emph{chain condition} \cite{WY1993}, or to be \emph{chained}, if it admits a complete filtration by minimum weight subcodes.

From \eqref{formule_def_deg_C'} and \eqref{defdi} we see that the canonical polygon of $C$ is entirely determined by its weight hierarchy,
or more precisely:
\begin{lemma}
\label{envelopeGWH}
The canonical polygon of $C$ is the upper convex envelope of the set of points
\beqv
\{(i,n-d_i(C)):\;0\leq i\leq k\}.
\eeqv
\end{lemma}

In particular, the vertices of the canonical polygon form a subset of this set of points, and the subcodes in the canonical filtration of $C$ are minimum weight subcodes.
Thus, even if $C$ is not chained, our theory (especially Theorem~\ref{th_filtr}) allows to extract a subsequence of dimensions,
and corresponding minimum weight subcodes, that form a (partial) filtration, in a canonical way.
For instance, if $C$ is totally unstable, \ie if it admits $k$ different slopes, then it is chained.
However, this example is rather extreme: as already noted, on the contrary, good codes tend to be stable.
And it is possible for a code to be both stable and chained: such is the simplex code introduced a few paragraphs above, or also its affine counterpart, the Reed-Muller code of order~$1$.



\vspace{.5\baselineskip}

As in the setting of vector bundles on curves, or in that of euclidean or hermitian lattices over number fields, an important question is that of the behaviour of the canonical polygon
(and in particular, the preservation of semistability) under basic operations such as extension of scalars, duality, and tensor product.

The first one is almost immediate:
\begin{proposition}
The canonical polygon of $C$, its canonical filtration, and its slopes, are preserved by extension of scalars.
\end{proposition}
\begin{proof}
Invariance of the weight hierarchy of $C$ under extension of scalars.
\end{proof}


\section{Slopes and duality}
\label{ssd}

Recall that if $C\subset\F^n$ is a $[n,k]$-code, then its \emph{dual code} $C^\perp\subset\F^n$ is the $[n,n-k]$-code
defined as the orthogonal of $C$ relative to the standard bilinear scalar product on $\F^n$.

\begin{remark}
At first, the word ``dual'' might seem misused here, however it has been well established historically.
Moreover, $C^\perp$ enjoys certain properties that make this terminology better suited than one could have first thought, especially in our context.
For instance, say in the binary case, we observe that if $\Gamma_C\subset\R^n$ is the euclidean lattice obtained from $C$ by the so-called Construction~A
(of \cite[\S5.2]{CS}, but suitably normalized, as in \cite[\S1.3]{Ebeling}), then $\Gamma_C$ and $\Gamma_{C^\perp}$ are dual lattices in the usual sense.
\end{remark}

In this section we will give the relation between the canonical polygon of a code $C$ and that of its dual code $C^\perp$.
Thanks to Lemma~\ref{envelopeGWH}, this could be deduced from Wei's duality theorem \cite{Wei1991}, which gives the relation between the weight hierarchy of $C$ and that of $C^\perp$.
However, it turns out it is slightly easier to work first from the other side of our cosupport Galois connection \eqref{cosupport_connection_C}\eqref{cosupport_connection_n}, in terms of matroids.
Actually, our proof does not require the matroid to come from a code, so what we will first state is a duality result for slopes of general matroids.
This will then be translated back to codes thanks to Proposition~\ref{symGal}.

A remarkable fact is that, although the proof uses only arguments from elementary linear algebra, these can be reinterpreted as part of some sort of Riemann-Roch theory, as developed in the Appendix at the end of this work.

\vspace{.5\baselineskip}

We say that a finite matroid $\cM=(E,\cI)$ is a $[n,k]$-matroid if it has cardinality $\#E=n$ and rank $r(E)=k$.
Recall \cite[\S2.1]{Oxley} that the \emph{dual matroid} $\cM^*$ of $\cM$ is then the $[n,n-k]$-matroid with the same underlying set $E$ but whose bases (maximal independent sets) are the complements of those of $\cM$.

For instance, the matroid $\cM$ associated with a $[n,k]$-code $C$ is a $[n,k]$-matroid, and its dual matroid $\cM^*$ is then the matroid associated with the dual code $C^\perp$.




\begin{theorem}
\label{thdualslopes}
Let $\cM=(E,\cI)$ be a $[n,k]$-matroid.
Suppose that, relative to $\cM$ and the associated degree function~\eqref{degmatroid}, the lattice $L_E$ has canonical polygon
\beq
\label{polygonDLPmatroid}
P_{\cM}:[0,n]\longto\R,
\eeq
with slopes
\beqv
\mu_1>\dots>\mu_N
\eeqv
(in the real interval $[-1,0]$).
Then, relative to the dual matroid $\cM^*$, the lattice $L_E$ has canonical polygon
\beq
\label{polygonDLPdualmatroid}
\begin{array}{cccc}
P_{\cM^*}: & [0,n] & \longto & \R\\
& x & \mapsto & P_{\cM}(n-x)+n-x-k,
\end{array}
\eeq
with slopes
\beqv
-1-\mu_N>\dots>-1-\mu_1.
\eeqv

Moreover, if the vertices of \eqref{polygonDLPmatroid} correspond to the canonical filtration
\beq
\label{filtrationDLPmatroid}
\emptyset=J_0\subsetneq J_1\subsetneq\dots\subsetneq J_N=E,
\eeq
then those of \eqref{polygonDLPdualmatroid} correspond to the dual filtration
\beq
\label{filtrationDLPdualmatroid}
\emptyset=(E\moins J_N)\subsetneq\dots\subsetneq(E\moins J_1)\subsetneq(E\moins J_0)=E.
\eeq
\end{theorem}
\begin{proof}
By definition the canonical polygon of $\cM$ is the upper convex envelope of the set of points
\beqv
(\#J,\deg(J))
\eeqv
for $J\in L_E$, where by~\eqref{degmatroid} we have $\deg(J)=k-r(J)$.
Hence it is also the the upper convex envelope of the set of points
\beqv
(j,f_\cM(j))
\eeqv
where for $0\leq j\leq n$ we set
\beqv
f_\cM(j)=\max\{\deg(J):\;J\subset E,\,\#J=j\}.
\eeqv
Likewise the canonical polygon of $\cM^*$ is the upper convex envelope of the set of points
\beqv
(j,f_{\cM^*}(j))
\eeqv
where $f_{\cM^*}(j)=\max\{\deg^*(J):\;J\subset E,\,\#J=j\}$, $\;\deg^*(J)=n-k-r^*(J)$.

Now a basic duality result for matroids \cite[Prop.~2.1.9]{Oxley} reads
\beqv
r^*(J)=\#J-k+r(E\moins J),
\eeqv
which we can also write as
\beqv
\deg^*(J)=\deg(E\moins J)+n-\#J-k.
\eeqv
Letting $J$ vary in this last equality, with $\#J=j$ fixed, we then find that the maximum is attained in the left-hand and right-hand side for the same $J$, with value
\beqv
f_{\cM^*}(j)=f_{\cM}(n-j)+n-j-k
\eeqv
(which is actually Proposition~\ref{dualDLPmatroid} in the Appendix).
The conclusion follows.
\end{proof}

\begin{corollary}
The matroid $\cM$ is semistable if and only if its dual $\cM^*$ is.
\end{corollary}



\begin{corollary}
\label{cordualslopes}
Suppose that $C$ has both minimum distance and dual distance at least $2$, \ie that both $C$ and $C^\perp$ have full support.
Let then $C$ have canonical filtration
\beqv
0=C_0\subsetneq C_1\subsetneq\dots\subsetneq C_N=C
\eeqv
with slopes
\beqv
\mu_1>\dots>\mu_N
\eeqv
(in $\R_{<-1}$).
Then $C^\perp$ has canonical filtration
\beqv
0=(C^\perp)_{[n]\moins\Supp(C_N)}\subsetneq\dots\subsetneq (C^\perp)_{[n]\moins\Supp(C_1)}\subsetneq (C^\perp)_{[n]\moins\Supp(C_0)}=C^\perp
\eeqv
with slopes
\beqv
-1+(\mu_N+1)^{-1}>\dots>-1+(\mu_1+1)^{-1}.
\eeqv
\end{corollary}
\begin{proof}
Let $\cM$ be the matroid associated with $C$, so $\cM^*$ is the matroid associated with $C^\perp$.
From the canonical filtration and slopes of $C$, we can deduce those of $\cM$ thanks to Proposition~\ref{symGal}.
We then deduce those of $\cM^*$ thanks to Theorem~\ref{thdualslopes}, and finally we deduce those of $C^\perp$ back thanks to Proposition~\ref{symGal} again.
\end{proof}

From this last result, the case where $C$ or $C^\perp$ does not have full support is also easily treated,
\eg thanks to the ``standard decomposition'' \cite[eq.~(2)]{Assmus1998},
and our Lemma~\ref{lemsepGal} (or also Proposition~\ref{slope=-rate-1}).
For instance, we see that $C^\perp$ does not have full support if and only if $\mumax(C)=-1$,
in which case the first step of the canonical filtration of $C$ is the subcode $V\subset C$ generated by the codewords of weight~$1$.

\begin{corollary}
A linear code $C$ is semistable if and only if its dual $C^\perp$ is.
\end{corollary}


\section{Semistability and tensor product}

Given a $[n_A,k_A]$-code $A\subset\F^{n_A}$ and a $[n_B,k_B]$-code $B\subset\F^{n_B}$, we can form their tensor product
\beqv
A\tens B\quad\subset\quad\F^{n_A}\tens\F^{n_B}\simeq\F^{n_An_B}
\eeqv
which is thus a $[n_An_B,k_Ak_B]$-code, of rate $\ratef(A\tens B)=\ratef(A)\ratef(B)$.

From the geometric point of view, if $A$ corresponds to a collection of $n_A$ points $\Pi_A\subset\PP^{k_A-1}$,
and $B$ to a collection of $n_B$ points $\Pi_B\subset\PP^{k_B-1}$,
then $A\tens B$ corresponds to the image of $\Pi_A\times\Pi_B$ under the Segre embedding $\PP^{k_A-1}\times\PP^{k_B-1}\to\PP^{k_Ak_B-1}$.

It is easily shown that if $A$ has minimum distance $d_1(A)=d_A$ and $B$ has minimum distance $d_1(B)=d_B$,
then $A\tens B$ has minimum distance $d_1(A\tens B)=d_Ad_B$.
An interesting problem is the derivation of similar estimates for the higher weights $d_r(A\tens B)$.
In \cite{Schaathun2000}, Schaathun proved an important lower bound on these quantities.

For $0\leq r\leq k_Ak_B$, set
\beqv
\begin{split}
d_r^*(A\tens B)&=\min\left\{\sum_{i=1}^s(d_i(A)-d_{i-1}(A))d_{t_i}(B):\right.\\
&\qquad\qquad\qquad\qquad\qquad \left.1\leq t_s\leq\cdots\leq t_1\leq k_B,\;\; s\leq k_A,\;\;\sum_{i=1}^s t_i=r\right\}\\
&= \min\left\{\sum_{i=1}^{k_A}(d_i(A)-d_{i-1}(A))d_{t_i}(B):\right.\\
&\qquad\qquad\qquad\qquad\qquad \left.0\leq t_{k_A}\leq\cdots\leq t_1\leq k_B,\;\;\sum_{i=1}^{k_A} t_i\geq r\right\}\\ 
\end{split}
\eeqv
(observe that the minimum in the second line is the same as in the first: this is because the $d_i(A)$ and $d_i(B)$ form monotonously increasing sequences). 

\begin{theorem}[\cite{Schaathun2000}]
\label{Schaathun}
With the notations above, for any subcode $C\subset A\tens B$ of dimension $\dim(C)=r$, we have $w(C)\geq d_r^*(A\tens B)$.
Or equivalently,
\beqv
d_r(A\tens B)\geq d_r^*(A\tens B).
\eeqv
\end{theorem}
Previously, in \cite{WY1993}, Wei and Yang had proved inequality in the other direction when both $A$ and $B$ are chained
(hence Schaathun's bound implies equality in this case).

Thanks to Lemma~\ref{envelopeGWH}, a lower bound on the weight hierarchy is the same as an upper bound on the canonical polygon.
It is then no surprise that Schaathun's bound will be the key ingredient to our:

\begin{theorem}
\label{produit_sst}
Suppose the linear codes $A$ and $B$ are semistable. Then their tensor product $A\tens B$ is semistable.
\end{theorem}

The analogue of Theorem~\ref{produit_sst} is known to be true for vector bundles over curves in characteristic zero,
but false in positive characteristic, where counter-examples have been constructed \cite{Gieseker1973}.
For euclidean lattices (or more generally hermitian lattices, or adelic vector bundles over number fields) the question has been popularized by Bost and is still open,
despite recent progress \cite{BC2013}\cite{GR2013} that allow to settle low dimensional cases, or lattices admitting a large group of automorphisms.

Schaathun's proof of Theorem~\ref{Schaathun} was written in terms of systems of points in a projective space, using the geometric view on coding theory.
We propose here a reformulation directly in terms of codes.

Because of the close links between lattices of codes, it is not unreasonable to hope that this reformulation could provide inspiration for new approaches on the semistability conjecture for tensor products of lattices.
Moreover the proof we will give here will appear to be shorter and marginally simpler than Schaathun's original one, although closer inspection would reveal the ideas remain fundamentally the same.

First we introduce some notation. For any $j\in[n_B]$, let $\pi_j:\F^{n_B}\to\F$ be the $j$-th standard projection. Tensoring with $\F^{n_A}$ then gives us
\beqv
\widehat{\pi_j}:\F^{n_A}\tens\F^{n_B}\longto\F^{n_A}.
\eeqv
It is customary to identify $\F^{n_A}\tens\F^{n_B}$ with the space $\F^{n_A\times n_B}$ of matrices with $n_A$ rows and $n_B$ columns.
The projection $\widehat{\pi_j}$ is then just the linear map that associates to each such matrix its $j$-th column.
Under this identification, $A\tens B$ becomes the space of matrices all of whose columns are in $A$ and all of whose rows are in $B$.
Thus, for any subcode $C\subset A\tens B$, we see that $\widehat{\pi_j}(C)$ is a subcode of $A$.
Moreover, the support size of $C$ can be computed as
\beq
\label{support_colonnes}
w(C)=\sum_{j=1}^{n_B}w(\widehat{\pi_j}(C)).
\eeq

In order to prove Theorem~\ref{Schaathun}, it suffices, given any $C\subset A\tens B$, to construct a sequence $0\leq t_{k_A}\leq\cdots\leq t_1\leq k_B$
such that $w(C)\geq\sum_{i=1}^{k_A}(d_i(A)-d_{i-1}(A))d_{t_i}(B)$ and $\sum_{i=1}^{k_A} t_i\geq\dim(C)$.
Thus Theorem~\ref{Schaathun} follows from the following more precise result:

\begin{lemma}
\label{lemSchaathun}
Given a subcode $C\subset A\tens B$, set, for $1\leq i\leq k_A$,
\beqv
J_i=\{j\in[n_B]:\;\dim(\widehat{\pi_j}(C))\geq i\}
\eeqv
and
\beqv
t_i=\dim(B_{J_i}).
\eeqv
Then we have
\beq
\label{estpoids}
w(C)\geq\sum_{i=1}^{k_A}(d_i(A)-d_{i-1}(A))d_{t_i}(B)
\eeq
and
\beq
\label{estdim}
\sum_{i=1}^{k_A} t_i\geq\dim(C).
\eeq
\end{lemma}
\begin{proof}[Proof (of Lemma~\ref{lemSchaathun}, hence also of Theorem~\ref{Schaathun}).]
For $1\leq i\leq k_A$ we have, by the very definition of higher weights, 
\beq
\label{minA}
w(\widehat{\pi_j}(C))\geq d_i(A)\qquad\text{for $j\in J_i$}
\eeq
and
\beq
\label{minB}
\#J_i\geq w(B_{J_i})\geq d_{t_i}(B).\qquad\qquad
\eeq
This works also for $i=0$ and $i=k_A+1$ if we set $J_0=[n_B]$, $t_0=k_B$, and $J_{k_A+1}=\emptyset$, $t_{k_A+1}=0$.

The decreasing filtration $[n_B]=J_0\supset J_1\supset\cdots\supset J_{k_A}\supset J_{k_A+1}=\emptyset$ then
allows to rewrite \eqref{support_colonnes} as
\begin{align*}
w(C)&=\sum_{i=0}^{k_A}\sum_{j\in J_i\moins J_{i+1}}w(\widehat{\pi_j}(C))\\
&\geq\sum_{i=0}^{k_A}d_i(A)(\#J_i-\#J_{i+1}) &&\text{by \eqref{minA}}\\
&=\sum_{i=1}^{k_A}(d_i(A)-d_{i-1}(A))\#J_i &&\text{summation by parts}\\
&\geq\sum_{i=1}^{k_A}(d_i(A)-d_{i-1}(A))d_{t_i}(B) &&\text{by \eqref{minB},}
\end{align*}
which proves~\eqref{estpoids}.

Now for any $i$, choose a relative information set $S_i$ for $B_{J_i}$ modulo $B_{J_{i+1}}$, \ie a subset
\beqv
S_i\subset J_i\moins J_{i+1}
\eeqv
such that the projection $\pi_{S_i}$ induces a bijection
\beqv
B_{J_i}/B_{J_{i+1}}\overset{\sim}{\longto}\,\F^{S_i}
\eeqv
(this is possible since, by definition, $B_{J_i}=B\cap\F^{J_i}\subset\F^{J_i}$ so, passing to the quotient, $\pi_{J_i\moins J_{i+1}}$ induces an injection $B_{J_i}/B_{J_{i+1}}\inj\F^{J_i\moins J_{i+1}}$).
Then necessarily, $\#S_i=\dim(B_{J_i}/B_{J_{i+1}})=t_i-t_{i+1}$.

Considering the filtration $B=B_{J_0}\supset B_{J_1}\supset\cdots\supset B_{J_{k_A}}\supset B_{J_{k_A+1}}=0$,
we see that the disjoint union $S=S_0\sqcup S_1\sqcup\cdots\sqcup S_{k_A}$ is an information set for $B$,
\ie we have an isomorphism
\beqv
\pi_S:B\overset{\sim}{\longto}\,\F^S.
\eeqv
Tensoring with $A$ then gives $\widehat{\pi_S}:A\tens B\overset{\sim}{\longto}\,A^S$, which restricts to $C$ as
\beqv
C\simeq\widehat{\pi_S}(C)\;\;\subset\;\;\bigoplus_{j\in S}\widehat{\pi_j}(C)
\eeqv
(or said more concretely: a codeword in $B$ is determined by its coordinates over $S$, hence a codeword in $A\tens B$, thus \emph{a fortiori} a codeword in $C$,
is determined by its columns over $S$).

Now writing $\bigoplus_{j\in S}\widehat{\pi_j}(C)=\bigoplus_{i=0}^{k_A}\bigoplus_{j\in S_i}\widehat{\pi_j}(C)$, and using the fact that $\dim(\widehat{\pi_j}(C))=i$ for $j\in S_i\subset J_i\moins J_{i+1}$,
we conclude
\beqv
\dim(C)\leq\sum_{i=0}^{k_A}\sum_{j\in S_i}\dim(\widehat{\pi_j}(C))=\sum_{i=0}^{k_A}i\#S_i=\sum_{i=0}^{k_A}i(t_i-t_{i+1})=\sum_{i=1}^{k_A}t_i.
\eeqv
\end{proof}

We can now proceed to our proof of the semistability theorem for tensor products of codes.
This can be derived either as a consequence of Theorem~\ref{Schaathun}, or slightly more directly,
by reworking part of the proof of Lemma~\ref{lemSchaathun}. We propose the latter:

\begin{proof}[Proof of Theorem~\ref{produit_sst}.]
Let $C\subset A\tens B$ be any subcode, and let then the $J_i$ and the $t_i$ be given by Lemma~\ref{lemSchaathun}.
Thanks to \eqref{poidsstable}, semistability of $A$ implies
\beq
\label{Asst}
\begin{split}
w(\widehat{\pi_j}(C))&\geq\dim(\widehat{\pi_j}(C))/\ratef(A)\\
&\geq i/\ratef(A)\qquad\quad\text{for $j\in J_i$},
\end{split}
\eeq
and semistability of $B$ implies
\beq
\label{Bsst}
\#J_i\geq w(B_{J_i})\geq t_i/\ratef(B).
\eeq
The computation in the first half of the proof of Lemma~\ref{lemSchaathun} then becomes:
\begin{align*}
w(C)&=\sum_{i=0}^{k_A}\sum_{j\in J_i\moins J_{i+1}}w(\widehat{\pi_j}(C))\\
&\geq\frac{1}{\ratef(A)}\sum_{i=0}^{k_A}i(\#J_i-\#J_{i+1}) &&\text{by \eqref{Asst}}\\
&=\frac{1}{\ratef(A)}\sum_{i=1}^{k_A}\#J_i &&\text{summation by parts}\\
&\geq\frac{1}{\ratef(A)\ratef(B)}\sum_{i=1}^{k_A}t_i\geq\frac{\dim(C)}{\ratef(A\tens B)} &&\text{by \eqref{Bsst} and \eqref{estdim}.}
\end{align*}
Using \eqref{poidsstable} again, this means precisely that $A\tens B$ is semistable.
\end{proof}

\begin{remark}
We have shown that semistability is preserved under many natural operations on codes, such as extension of scalars, duality, and tensor product.
Here is an operation that does not preserve it: although \emph{componentwise product} of codes is closely related to tensor product \cite[\S1.10]{Randriambololona2015}, it is easily checked (\eg using Proposition~\ref{slope=-rate-1}) that the binary $[5,2]$-code $C$, with generator matrix
\beqv
\left(\begin{array}{ccccc}
1&0&0&1&1\\
0&1&1&1&1
\end{array}\right)
\eeqv
is semistable, and even stable, while its square $C^{\langle2\rangle}$, 
with generator matrix
\beqv
\left(\begin{array}{ccccc}
1&0&0&0&0\\
0&1&1&0&0\\
0&0&0&1&1
\end{array}\right)
\eeqv
is unstable.
\end{remark}

\begin{remark}
In Remark~\ref{remBost} we saw that codes could be seen as multi-filtered spaces in the sense of Faltings-W\"ustholz and Totaro, in a way that preserves the canonical filtration.
Moreover these authors also proved that, in characteristic zero, a tensor product of semistable multi-filtered spaces is semistable (in \cite[Th.~4.1]{FW1994} this is done by reducing to the case of vector bundles on curves, and in \cite{Totaro1994} this is done in a somehow more elementary way, although still relying on the GIT machinery).

However one should be careful: \emph{the tensor product of codes does not coincide with the tensor product of multi-filtered spaces!}

A first point is that \cite{FW1994}\cite{Totaro1994} consider two spaces with the same number $n$ of filtrations, and produce a tensor product also with $n$ filtrations.
From this perspective, the theory is more similar to that of componentwise products of the preceding remark, than to that of tensor products, where $A$ has $n_A$ filtrations, $B$ has $n_B$, and $A\tens B$ has $n_An_B$.

This first objection is easily bypassed.
Indeed, let $A$ be equipped with the $n_A$ filtrations $A\supset H_i$, and $B$ with the $n_B$ filtrations $B\supset K_j$.
Then we can repeat $n_B$ times each $H_i$, and repeat $n_A$ times each $K_j$.
Doing so, $A$ and $B$ get the same number $n_An_B$ of filtrations, and their tensor product as multi-filtered spaces is then equipped with the $n_An_B$ filtrations
\beqv
A\tens B \;\supset\; (H_i\tens B)+(A\tens K_j) \;\supset\; H_i\tens K_j
\eeqv
of length $2$.

However, the multi-filtered space associated with the tensor product code only corresponds to the \emph{truncated} filtrations
\beqv
A\tens B \;\supset\; (H_i\tens B)+(A\tens K_j)
\eeqv
of length~$1$.
In matrix form, codewords in $(H_i\tens B)+(A\tens K_j)$ are those vanishing at position $(i,j)$, as expected, while codewords in $H_i\tens K_j$ are those whose whole $i$-th row and $j$-th column are identically zero.

Thus our result and the one in \cite{FW1994}\cite{Totaro1994} are of a quite different nature. It is unclear whether one could be deduced from the other.
Still it would be interesting to try to adapt the method of proof of our Theorem~\ref{produit_sst}, which holds in arbitrary characteristic, to tensor products of multi-filtered spaces.
\end{remark}

\newpage
\renewcommand\refname{References (main)}


\newpage
\section*{Appendix: A Riemann-Roch theory for linear codes and matroids}
\addcontentsline{toc}{section}{Appendix: A Riemann-Roch theory for linear codes and matroids}

\vspace{.5\baselineskip}
\mysubsection{Subspace pair cohomology}

We introduce cohomology groups associated to triples $(V,A_1,A_2)$ where:
\begin{itemize}
\item $V$ is a vector space,
\item $A_1,A_2\subset V$ are two linear subspaces. 
\end{itemize}

\begin{definition}
\label{defHtriple}
Given such a triple $(V,A_1,A_2)$ we set
\beq
\label{H0triple}
H^0(V;A_1,A_2)=A_1\cap A_2
\eeq
and
\beq
\label{H1triple}
H^1(V;A_1,A_2)=V/(A_1+A_2),
\eeq
of dimension, respectively:
\beqv
h^0(V;A_1,A_2)=\dim(H^0(V;A_1,A_2)),\quad h^1(V;A_1,A_2)=\dim(H^1(V;A_1,A_2)).
\eeqv
\end{definition}

It turns out these $H^0$ and $H^1$ can be identified with the cohomology groups of any of the following two quasi-isomorphic complexes of length $2$:
\beq
\label{complexsub}
0\longto A_1\oplus A_2\longto V\longto 0,
\eeq
or
\beq
\label{complexquot}
0\longto V\longto V/A_1\oplus V/A_2\longto 0.
\eeq

\begin{remark}
\label{signes}
To be more precise, the exact definition of these complexes, together with a quasi-isomorphism between them, and the identification of \eqref{H0triple} and \eqref{H1triple} with their cohomology groups, all depend on some choice of signs.
A possible choice is the following:
\beqv
\label{qis}
\begin{tikzcd}
0 \arrow{r}{} &[-1.2em] A_1\!\cap\! A_2 \arrow[equal]{d}{} \arrow{r}{++\;} &[.8em] A_1\!\oplus\! A_2 \arrow{d}{i_1} \arrow{r}{+-\;} & \quad V\quad \arrow{d}{p_2} \arrow{r}{} & V\!/\!(\!A_1\!\!+\!\!A_2\!) \arrow[equal]{d}{} \arrow{r}{} &[-1.2em] 0\\
0 \arrow{r}{} &[-1.2em] A_1\!\cap\! A_2 \arrow{r}{} &[.8em] \;V\; \arrow{r}{-+\;} & V\!/\!A_1\!\oplus\! V\!/\!A_2 \arrow{r}{++\;} & V\!/\!(\!A_1\!\!+\!\!A_2\!) \arrow{r}{} &[-1.2em] 0
\end{tikzcd}
\eeqv
where the maps are defined as follows:
on the top line, the left $++$ map sends $x\in A_1\cap A_2$ to $(x,x)\in A_1\oplus A_2$, the middle $+-$ map sends $(a_1,a_2)\in A_1\oplus A_2$ to $a_1-a_2\in V$, and the right map is natural projection;
on the bottom line, the left map is natural inclusion, the middle $-+$ maps sends $v\in V$ to (the class of) $(-v,v)\in V/A_1\oplus V/A_2$,
and the right $++$ map sends (the class of) $(v_1,v_2)\in V/A_1\oplus V/A_2$ to (the class of) $v_1+v_2\in V/(A_1+A_2)$;
then the vertical $i_1$ sends $(a_1,a_2)\in A_1\oplus A_2$ to $a_1\in V$, and the vertical $p_2$ sends $v\in V$ to (the class of) $(0,v)\in V/A_1\oplus V/A_2$.\footnote{It is easily checked this makes everything commute, for instance,
the middle square maps $(a_1,a_2)\in A_1\oplus A_2$ to (the class of) $(0,a_1)\in V/A_1\oplus V/A_2$.}
\end{remark}

\begin{lemma}
\label{preRR}
Let $(V,A_1,A_2)$ be a triple as above. Then, if $V$ is finite dimensional, we have
\beqv
h^0(V;A_1,A_2)-h^1(V;A_1,A_2)=\dim(A_1)+\dim(A_2)-\dim(V).
\eeqv
\end{lemma}
\begin{proof}
Direct from Definition~\ref{defHtriple}, or alternatively, Euler characteristic of the complex \eqref{complexsub} (or of \eqref{complexquot}).
\end{proof}

\begin{lemma}
\label{les}
Let $(V,A_1,A_2)$ be a triple as above. Then for any linear subspace $A'_2\subset A_2$ we have a long exact sequence
\beqv
\begin{split}
0\longto H^0(V;A_1,A'_2)\longto H^0&(V;A_1,A_2)\longto A_2/A'_2\longto\dots\\
&\dots\longto H^1(V;A_1,A'_2)\longto H^1(V;A_1,A_2)\longto 0.
\end{split}
\eeqv
\end{lemma}
\begin{proof}
Direct from Definition~\ref{defHtriple}, or alternatively, snake lemma for
\beqv
\begin{tikzcd}
0 \arrow{r} & A_1\oplus A'_2 \arrow{r} \arrow{d} & A_1\oplus A_2 \arrow{r} \arrow{d} & A_2/A'_2 \arrow{r} \arrow{d} & 0\\
0 \arrow{r} & V \arrow{r} & V \arrow{r} & 0 \arrow{r} & 0
\end{tikzcd}
\eeqv
\end{proof}

Recall that $V^\vee$ denotes the dual vector space of $V$, \ie the space of linear forms on $V$.
Given a linear subspace $A\subset V$, we let $A^\perp\subset V^\vee$ be the space of linear forms vanishing on $A$.
We have natural identifications $A^\perp=(V/A)^\vee$ and $V^\vee/A^\perp=A^\vee$.

\begin{lemma}
\label{preSerre}
Given a triple $(V,A_1,A_2)$ as above, we have a natural identification
\beqv
H^1(V;A_1,A_2)=H^0(V^\vee;A_1^\perp,A_2^\perp)^\vee.
\eeqv
\end{lemma}
\begin{proof}
Direct from Definition~\ref{defHtriple}, or alternatively, from the fact that dualizing exact sequence \eqref{complexsub} for $(V;A_1,A_2)$ gives
exact sequence \eqref{complexquot} for $(V^\vee;A_1^\perp,A_2^\perp)$ (actually, with sign reversed compared to Remark~\ref{signes}, but this does not change its cohomology).
\end{proof}

\vspace{.5\baselineskip}
\mysubsection{Serre duality and Riemann-Roch for codes}

Let $C\subset\F^n$ be a $[n,k]$-code.
We will apply the formalism just introduced to the following triple:
\begin{itemize}
\item $V=\F^n\;$ somehow seen as an ``ad\`ele space''\footnote{Although we will not really need it, we observe that our ad\`ele space is in fact an ad\`ele \emph{ring},
and even more precisely a $\F$-algebra, thanks to componentwise multiplication. 
The standard scalar product, for which $C^\perp$ is the orthogonal of $C$, is then nothing but the canonically associated trace bilinear form \cite[\S2.37]{Randriambololona2015}.}
\item $A_1=C\;$ the subspace of principal ad\`eles
\item $A_2=\F^J\;$ the subspace of ad\`eles with poles in a subset $J\subset[n]$
(where we actually identify $\F^J$ with $\F^J\times\{0\}^{[n]\moins J}$).
\end{itemize}

This will allow us to mimic Weil's adelic proof of the Riemann-Roch theorem, as modern coding theorists often learn, for instance, in \cite{Stichtenoth}.

Slightly changing notations, we thus get the following cohomology groups:
\beq
\begin{split}
H^0(C,J)&=H^0(\F^n;C,\F^J)\\
&=C\cap\F^J\\
&=C_J,
\end{split}
\eeq
the largest (a.k.a. ``shortened'') subcode of $C$ with support in $J$, and
\beq
\label{H1}
\begin{split}
H^1(C,J)&=H^1(\F^n;C,\F^J)\\
&=\F^n/(C+\F^J)\\
&=\F^{[n]\moins J}/\pi_{[n]\moins J}(C),
\end{split}
\eeq
where $\pi_{[n]\moins J}:\F^n\to\F^{[n]\moins J}$ is the natural projection.
The author does not know whether this $H^1$ has ever been considered in the literature.
Its elements can be seen as syndromes, so maybe it could be of use, for instance, in some decoding algorithms.

Lemma~\ref{les} provides, for any disjoint $J,J'\subset[n]$, a long exact sequence
\beq
\label{les2}
0\to H^0(C,J)\to H^0(C,J\sqcup J')\overset{\pi_{J'}}{\longto}\,\F^{J'}\to H^1(C,J)\to H^1(C,J\sqcup J')\to0
\eeq
where we used the $\sqcup$ symbol as a reminder of the fact that the union is disjoint.
Of course this readily extends to possibly non-disjoint $J,J'$ using the equality $J\cup J'=J\sqcup(J'\moins J)$.

As before we also set $h^0(C,J)=\dim H^0(C,J)$ and $h^1(C,J)=\dim H^1(C,J)$.

\begin{theorem}
\label{Serre}
Let $C\subset\F^n$ be a $[n,k]$-code.
Then for any subset $J\subset[n]$ we have a canonical identification
\beqv
H^1(C,J)=H^0(C^\perp,[n]\moins J)^\vee.
\eeqv
\end{theorem}
\begin{proof}
This is Lemma~\ref{preSerre}.
\end{proof}

\begin{theorem}
\label{RR}
Let $C\subset\F^n$ be a $[n,k]$-code.
Then for any subset $J\subset[n]$ we have the equality
\beqv
h^0(C,J)-h^0(C^\perp,[n]\moins J)\,=\,\#J\,+k-n.
\eeqv
\end{theorem}
\begin{proof}
Combine Lemma~\ref{preRR} with Theorem~\ref{Serre}.
\end{proof}

Since $h^0(C,J)$ vanishes for $\#J<d=\dmin(C)$, it is convenient to introduce the normalized cardinality $\abs{J}_{\textrm{norm}}=\#J-d$.
Theorem~\ref{RR} then becomes:
\beq
\label{RRnorm}
h^0(C,J)-h^0(C^\perp,[n]\moins J)\,=\,\abs{J}_{\textrm{norm}}+1-g
\eeq
where $g=n-k-d+1$ is the Singleton defect of $C$.

\vspace{.5\baselineskip}

We find it pleasant to think of Theorem~\ref{Serre} as a Serre duality theorem for codes, and of Theorem~\ref{RR}, or~\eqref{RRnorm}, as a Riemann-Roch theorem.
Seemingly, when $C$ is an AG code defined by evaluation of functions on an algebraic curve,
these results partially reflect Serre duality and Riemann-Roch on the underlying curve.
However we would like to stress that they hold for \emph{any} arbitrary linear code.

Of course a more direct proof could be given using only elementary linear algebra, for instance by combining lemmas from \cite{Forney1994}.
However our main purpose here is not really in the content of these theorems, nor in giving the most efficient proof, but rather on the analogy with the geometric situation. Thus our ``adelic'' proof was presented in order to make this analogy even stronger.

Observe that our statements do not involve only the code $C$ but also its dual $C^\perp$, and hence become more symmetric when applied to \emph{self-dual} codes.
Although it is not entirely clear what status self-dual codes enjoy under our chain of analogies (\eg ``$2$-torsion points'' in some ``Jacobian''?), they certainly have a great arithmetic significance.
For instance, Construction~A of \cite[\S5.2]{CS} (suitably normalized) transforms binary self-dual codes into unimodular integral euclidean lattices.
We refer to \cite{Ebeling} for more on this topic.
In a more anecdotal way, we observe that self-dual codes also satisfy the following \emph{Clifford}-like estimate:
\beq
h^0(C,J)\leq\frac{1}{2}\#J.
\eeq

It is interesting to note that, even if not explicitely stated so, the analogy of Theorems~\ref{Serre} and~\ref{RR} with Serre duality and Riemann-Roch was already implicit in previous works of other authors.
We point out especially \cite{Duursma2003}: as observed below, the $h^0$ and $h^1$ introduced there, p.~123, coincide with ours.
A consequence of this observation (and, in the first place, of the results in \cite{Duursma2003}) is that this Riemann-Roch theorem for codes implies the functional equation for the Duursma zeta function
essentially in the same way that the usual Riemann-Roch theorem implies the functional equation for the zeta function of curves over finite fields.

\vspace{.5\baselineskip}
\mysubsection{Generalization to matroids}

It turns out our Theorem~\ref{RR} admits a generalization, and also a more direct proof, in the context of matroids.

We keep notations as in section~\ref{ssd}
of the main text: a finite matroid $\cM=(E,\cI)$ is a $[n,k]$-matroid if it has cardinality $\#E=n$ and rank $r(E)=k$;
its dual matroid $\cM^*$ then is a $[n,n-k]$-matroid.

Now, following \cite{Duursma2003}, for any $J\subset E$, we set
\beq
\label{h0matroid}
h^0(\cM,J)=r(E)-r(E\moins J).
\eeq
Our Riemann-Roch theorem for matroids then reads:
\begin{proposition}
\label{RRmatroid}
Let $\cM=(E,\cI)$ be a $[n,k]$-matroid.
Then for any $J\subset E$ we have
\beqv
h^0(\cM,J)-h^0(\cM^*,E\moins J)\,=\,\#J\,+k-n.
\eeqv
\end{proposition}
\begin{proof}
Replacing with~\eqref{h0matroid}, $r(E)=k$, and $r^*(E)=n-k$, we are reduced to show $r^*(J)-r(E\moins J)=\#J-k$, which is nothing but \cite[Prop.~2.1.9]{Oxley}.
\end{proof}

Observe that if $\cM$ is the matroid with underlying set $[n]$, defined by the columns of the generating matrix of a $[n,k]$-code $C$, then for any $J\subset[n]$ we have
\beqv
h^0(\cM,J)=h^0(C,J).
\eeqv
Also the dual matroid $\cM^*$ corresponds in the same way to the dual code $C^\perp$.
Thus in this situation, Proposition~\ref{RRmatroid} reduces to Theorem~\ref{RR}.

In \cite{Duursma2003} one also finds the quantity
\beq
\label{h1matroid}
h^1(\cM,J)=\#(E\moins J)-r(E\moins J).
\eeq
From Proposition~\ref{RRmatroid} (or from \cite[Prop.~2.1.9]{Oxley}) we immediately see that it satisfies
\beqv
h^1(\cM,J)=h^0(\cM^*,E\moins J),
\eeqv
a weak form of Serre duality (equality of ``dimensions'', but no actual duality map as in Theorem~\ref{Serre}).

In case $\cM$ comes from a linear code $C$, we also have $h^1(\cM,J)=h^1(C,J)$ (direct consequence of~\eqref{H1}). So the two definitions of $h^1$ coincide, as did the $h^0$.

\vspace{.5\baselineskip}
\mysubsection{Wei's duality for higher weights}

We define the weight hierarchy and the dimension/length profile (DLP) of a matroid as the two ``dual'' sequences
\beqv
\label{dimatroid}
d_i(\cM)=\min\{\#J:\;J\subset E,\,h^0(\cM,J)=i\}\qquad(i=0\dots k)
\eeqv
and
\beqv
\label{DLPmatroid}
k_j(\cM)=\max\{h^0(\cM,J):\;J\subset E,\,\#J=j\}\qquad(j=0\dots n).
\eeqv
When $\cM$ comes from a code $C$, these are easily seen to coincide with those introduced in \cite{Wei1991} and \cite{Forney1994}.

Pursuing our geometric analogy, the weight hierarchy clearly corresponds to the gonality sequence of a curve $X$
\beqv
\gamma_i(X)=\min\{\deg(D):\;D\in\Div(X),\,h^0(X,D)=i\}\qquad(i>0),
\eeqv
an observation already made in \cite{Munuera1994}\cite{YKS1994}\cite{Duursma2003}.
Then in \cite{Duursma2003} there is also a derivation of Wei's duality theorem using the rank polynomial of the code.
Here we will proceed in a way closer to that of \cite{Forney1994}.
Observe that all our computations will be in the matroid setting, so what we will get in the end is a matroid version of Wei's duality. Actually, such a result is not new: it first appeared in \cite{BJMS2012}.

\vspace{.5\baselineskip}

As in \cite{Forney1994}, we first prove duality for the DLP.
Interestingly, this result is also at the heart of our Theorem~\ref{thdualslopes} on duality for slopes and canonical filtration.

\begin{proposition}
\label{dualDLPmatroid}
Let $\cM=(E,\cI)$ be a $[n,k]$-matroid. Then for any $j\in[n]$ we have
\beqv
k_{n-j}(\cM^*)=k_j(\cM)+n-j-k.
\eeqv
Moreover, if $k_j(\cM)=h^0(\cM,J)$ for some $J\subset E$ with $\#J=j$, then $k_{n-j}(\cM^*)=h^0(\cM^*,E\moins J)$.
\end{proposition}
\begin{proof}
For $J\subset E$ with $\#J=j$ we write Proposition~\ref{RRmatroid} as
\beq
h^0(\cM^*,E\moins J)=h^0(\cM,J)+n-j-k.
\eeq
Letting $J$ vary with $\#J=j$ fixed, we then find that the minimum is attained in the left-hand and right-hand side for the same $J$, with value $k_{n-j}(\cM^*)=k_j(\cM)+n-j-k$ as claimed.
\end{proof}

From the very definition, for any $J\subset E$ and $e\in E\moins J$, we have $r(J\cup\{e\})=r(J)$ or $r(J)+1$.
This translates as
\beq
h^0(\cM,J\cup\{e\})=h^0(\cM,J)\text{ or }h^0(\cM,J)+1
\eeq
(in case $\cM$ comes from a code $C$, this also follows from~\eqref{les2} with $J'=\{e\}$).

Plugging this into the definition of the DLP, it follows easily:
\begin{lemma}
For any integer $j\in[n]$ we have $k_j(\cM)=k_{j-\!1}(\cM)$ or $k_{j-\!1}(\cM)+1$.
\end{lemma}

\begin{definition}
Let $\cM$ be a $[n,k]$-matroid, and $j\in[n]$ an integer.
\begin{itemize}
\item We say that $j$ is a \emph{gap} for $\cM\,$ if $k_j(\cM)=k_{j-1}(\cM)$.
\item Else, we say $j$ is a \emph{non-gap} if $k_j(\cM)=k_{j-1}(\cM)+1$.
\end{itemize}
\end{definition}

Since $k_0(\cM)=0$ and $k_n(\cM)=k$, we deduce that $\cM$ admits precisely $n-k$ gaps and $k$ non-gaps.

\begin{lemma}
\label{gapnongap}
An integer $j\in[n]$ is a non-gap for $\cM$ if and only if $n+1-j$ is a gap for $\cM^*$ (and conversely).
\end{lemma}
\begin{proof}
Replacing in Proposition~\ref{dualDLPmatroid}, we see $k_j(\cM)=k_{j-1}(\cM)+1$ if and only if $k_{n-j}(\cM^*)=k_{n+1-j}(\cM^*)$.
\end{proof}

\begin{lemma}
\label{nongaps=GHW}
The non-gaps of $\cM$ coincide with its (nonzero) higher weigths $d_1(\cM),\dots,d_k(\cM)$.
\end{lemma}
\begin{proof}
That $j$ is a non-gap means $l=k_j(\cM)>k_{j-1}(\cM)=l-1$, or equivalently, there is a $J\subset E$ with $h^0(\cM,J)=l$ and $\#J=j$, but $h^0(\cM,J')<l$ for all $J'$ of cardinality $\#J'<j$.
In turn, this means precisely $d_l(\cM)=j$.
\end{proof}

We can now finish our proof of Wei's duality for higher weights of matroids:
\begin{proposition}[{\cite[Th.~1]{BJMS2012}}]
\label{Weimatroid}
Let $\cM$ be a $[n,k]$-matroid. Then the two sets of integers
\beqv
\{d_i(\cM):\;i\in[k]\}\qquad\text{and}\qquad\{n+1-d_i(\cM^*):\;i\in[n-k]\}
\eeqv
are disjoint, hence they form a partition of $[n]$.
\end{proposition}
\begin{proof}
Combine Lemma~\ref{gapnongap} and Lemma~\ref{nongaps=GHW}.
\end{proof}

\renewcommand\refname{Supplementary references (appendix)}

\end{document}